\documentclass[twocolumn]{autart}    

\usepackage{amsmath,amssymb,amsfonts}                               
\usepackage{mathrsfs}
\usepackage{color} 
\usepackage{apacite}

\usepackage{subfigure}
\usepackage{graphicx,epsfig}

\usepackage{algorithm}
\usepackage{algorithmicx}
\usepackage{algpseudocode}

\usepackage{appendix}

\newtheorem{definition}{Definition}
\newtheorem{theorem}{Theorem}

\newtheorem{Proposition}{Proposition}

\newtheorem{lemma}{Lemma}

\newtheorem{Remark}{Remark}

\newtheorem{Problem}{Problem}

\begin{document}

\begin{frontmatter}

\title{\thanksref{footnoteinfo}} 

\thanks[footnoteinfo]{This paper was not presented at any IFAC 
meeting.\\ $*$ Corresponding author.}

\author[Qingdao,Zibo]{Jing Guo}\ead{guojing8299@163.com},          
\author[East]{Xiushan Jiang}\ead{ x\_sjiang@163.com},	        
\author[Qingdao]{Weihai Zhang}$^{,*}$\ead{w\_hzhang@163.com}       

\address[Qingdao]{College of Electrical Engineering and Automation, Shandong University of Science and Technology, Qingdao 266590, China}                                               
\address[Zibo]{School of Mathematics and Statistics,
	Shandong University of Technology, Zibo 250049,
	China} 
\address[East]{College of New Energy, China University of Petroleum (East China), Qingdao 266580, China}                                                    
          

\begin{abstract}                          
	The stochastic $H_{\infty}$ control  is studied for a linear stochastic It\^o system with an unknown system model. The linear stochastic $H_{\infty}$ control issue is known to be transformable into the problem of solving a so-called generalized algebraic Riccati equation (GARE), which is a nonlinear equation that is typically difficult to solve analytically. Worse, model-based techniques cannot be utilized to approximately solve a GARE when an accurate system model is unavailable or prohibitively expensive to construct in reality. To address these issues, an off-policy reinforcement learning (RL) approach is presented to learn the solution of a GARE from real system data rather than a  system model; its convergence is demonstrated, and the robustness of RL to   errors in the learning process is investigated.
	In the off-policy RL approach, the system data may be created with  behavior policies  rather than the target policies, which is highly significant and promising for  use in actual systems. Finally, the proposed off-policy RL approach is validated on a stochastic linear F-16 aircraft system.
\end{abstract}

\end{frontmatter}

\section{Introduction}
Reinforcement learning (RL), which has its roots in animal learning psychology, is a method
that learns via trial-and-error and initially gained much interest in the field of artificial
intelligence. The link between RL approaches and control domains was established by Sutton
\cite{sutton1999reinforcement}. In the field of control, RL method refers to a controller's interaction with dynamical system, with the goal of learning optimal control policies by monitoring certain performance index,  without
the full knowledge of system dynamics \cite{sutton1999reinforcement,bertsekas2019reinforcement}.  Many RL approaches have been presented during the last decades for various optimal control problems associated with various dynamical systems, as discussed in  a recent  survey \cite{kiumarsi2017optimal}. The majority of existing RL approaches, however, are presented for stochastic discrete-time systems represented by Markov decision processes or deterministic continuous-time systems governed by ordinary differential equations. There are few results for stochastic continuous-time systems governed by stochastic differential equations, which are important in the
modeling of stochastic uncertainties in actual  systems. For stochastic optimal control problems, RL
approaches have been successfully applied. In \cite{bian2016adaptive}, the optimal control issue was addressed for a class of continuous-time stochastic
systems perturbed by multiplicative noise, and robust optimality
analysis was conducted.  \cite{li2022stochastic} presented an
online RL algorithm to solve infinite horizon continuous-time stochastic linear
quadratic problems  with partial system
information. \cite{pang2022reinforcement} proposed a novel off-policy  RL algorithm that can determine near-optimal policies for an optimal stationary control problem directly from data.  \cite{wei2023continuous} developed a new RL-based  method to solve optimal control problem for nonlinear systems with  stochastic nonlinear disturbances.

$H_{\infty}$ control is one of the most significant robust control
approaches and has received much attention in the last forty years, see \cite{zames1981feedback,doyal1989state,van19922,bacsar1995h,damm2002state}. $H_{\infty}$ control is used to attenuate the effects of external disturbances on the outputs, which is mathematically represented by the $H_{\infty}$ norm,  below a given disturbance attenuation level. In practice, systems are subject to various random noises both internally and externally. The uncertainty of system parameters is usually modeled as multiplicative noise, while some external perturbation is modeled as additive noise. In the framework of stochastic systems, the $H_{\infty}$ norm is exactly the  $\mathcal{L}_{2}$-induced norm of the input-output perturbation  operator with initial state zero \cite{hinrichsen1998stochastic}. For continuous-time linear systems, 
finding the solutions to the deterministic and stochastic $H_{\infty}$ control problems leads to solving the algebraic Riccati
equation (ARE) and the generalized algebraic Riccati equation (GARE), respectively. The first challenge in numerically solving  $H_{\infty}$ control  problem is that the quadratic terms in ARE and GARE are indefinite, which results in Kleinman's algorithm commonly used in solving optimal control problems no longer being applicable. \cite{lanzon2008computing} suggested an iterative technique for
solving ARE, in which the  ARE with an indefinite sign of the quadratic term is  transformed into a sequence of
AREs that can be solved by Kleinman's algorithm. 
\cite{feng2010iterative} and \cite{dragan2011computation} further expanded the approach in
\cite{lanzon2008computing} to solve stochastic $H_{\infty}$ control issue. 
\cite{wu2013simultaneous} provided a highly computationally efficient simultaneous policy
update (SPU) algorithm, in which the control and disturbance policies are updated simultaneously, and
they developed offline and online versions, which are model-based and partially model-free approaches,
respectively, to improve efficiency. However, the approaches described above mostly are intended for
continuous-time deterministic $H_{\infty}$ control, and the majority of them are model-based or partially
model-free; little study has been done on model-free algorithms for continuous-time stochastic $H_{\infty}$
control.

Compared to deterministic systems or systems with additive noise, systems with multiplicative noise are more enriching, and they are closely related to many complex systems that are difficult to model. The goal of this study is to give a novel approach to solve the GARE arising in 
stochastic $H_{\infty}$ control with state-dependent multiplicative noise. Unlike AREs arising from deterministic $H_{\infty}$ control, the
GAREs coming from stochastic $H_{\infty}$ control  have an extra linear disturbance term connected to the state-dependent multiplicative noise coefficient matrix. Because of the
extra disturbance term, a GARE is often more challenging to address than an ARE.
Based on the SPU algorithm proposed in \cite{wu2013simultaneous}, we design a model-based
SPU algorithm  to solve this GARE, which is shown to be a Newton's
algorithm, and then propose an off-policy RL algorithm  for solving a
GARE without knowing all of the system's information in advance. The following are the primary
contributions of this study.

1) The convergence of model-based SPU algorithm is proven using Kantorovich's Theorem by
proving that it is equivalent to Newton's algorithm, and model-based  algorithm is
shown to have local stepwise stabilizability and a local quadratic convergence rate using the mean square stable spectral criterion and the $\mathscr{H}$-representation technique.

2) RL algorithm is an off-policy method that does not need specified updates of the
disturbance policy, which is sensible in actual situations. It is also a model-free method
that does not require complete system knowledge. Furthermore, we provide a formal mathematical
demonstration of RL algorithm's convergence under the rank condition without considering errors due
to random noise. In addition, we investigate the robustness of the off-policy method to
errors in the learning process in the context of bias owing to random noise. This contrasts with the
robustness of controllers learned by RL to dynamic perturbations in the system \cite{bian2016adaptive}.
It can be shown that if the procedure begins with a solution that is near the optimal
one and the errors are small enough, the differences between the solutions obtained by the
algorithm and the true solution will be small and bounded as well.

The remainder of the paper is structured as follows: Section 2 describes the problem and the
fundamental conclusions for the stochastic $H_{\infty}$ control problem. Sections 3 and 4 
present the model-based SPU algorithm and the off-policy RL approach for addressing the $H_{\infty}$
control  for a linear stochastic It\^o system, respectively. Rigorous mathematical
proof
of convergence is also provided. Section 5  analyses the robustness of the off-policy method to
errors. Section 6 includes a simulation example to demonstrate the effectiveness of the suggested method, and
Section 7 concludes this paper with reviewing conclusions. \par
Notations: $\mathcal{R}_{+}\left(\mathcal{Z}_{+}\right) $ is the set of nonnegative real numbers (integers); $\mathcal{C}_{-}$ is the open left-hand side of the complex plane. $I_{n}$ denotes the identity matrix in $\mathcal{R}^{n \times n}$ while $0$ denotes the zero vector or matrix with the appropriate dimension; $\left\|\cdot\right\| _{2}$ denotes the Euclidean
norm for vectors and the spectral norm for matrices.  $\mathcal{S}^{n}$ and $\mathcal{S}^{n}_{+}$   are the sets of all $n \times n$ symmetric and symmetric positive semidefinite matrices respectively. For $X \in \mathcal{S}^{n}, Y \in \mathcal{S}^{n}$, $X \succeq Y$  denotes $X-Y\in \mathcal{S}^{n}_{+}$. For $H\in \mathcal{S}^{n}$, define $\operatorname{vecs}( H)=\left[h_{11}, \sqrt{2} h_{12}, \ldots, \sqrt{2} h_{1 n}, h_{22}, \sqrt{2} h_{23}, \ldots, \sqrt{2} h_{n-1, n}, h_{n n}\right]^{\top}$, where $h_{ik}$ is the $(i, k)$th element of matrix $H$. For vector $x \in \mathcal{R}^{n}$, define an operator $\tilde{x}=\operatorname{vecs}\left( xx^{\top}\right)\in \mathcal{R}^{\frac{n(n+1)}{2}}$.  For stochastic processes 
$\left\{x\left( t\right) \right\}_{t \in \mathcal{R}_{+}}$ which is defined on the complete probability 
space $\left(\Omega, \mathcal{F},\mathbb{P}\right)$, the $\mathcal{L}_{2}$-norm of $\left\{x\left( t\right) \right\}_{t \in \mathcal{R}_{+}}$ is defined as $\|x(\cdot)\|_{\mathcal{L}_{2}}=\left( \mathbb{E} \int_{0}^{\infty}\left( x(t)^{\top}x(t)\right) \mathrm{d} t\right)^{1 / 2}$. $\mathcal{L}_{\mathcal{F}}^{2}\left(\mathcal{R}_{+}; \mathcal{R}^{n}\right)$ denotes the space of nonanticipative $\mathcal{R}^{n}$-valued stochastic processes
$\left\{x\left( t\right) \right\}_{t \in \mathcal{R}_{+}}$ defined on the probability 
space $\left(\Omega, \mathcal{F},\mathbb{P}\right)$ with respect to an increasing $\sigma$- algebra $\left\{\mathcal{F}_{t}\right\}_{t \geq 0}$ satisfying $\mathbb{E} \int_{0}^{\infty}\left( x(t)^{\top}x(t)\right)  \mathrm{d} t<\infty$.
\section{Problem formulation and preliminaries}
In this section, we  describe the problem and highlight some preliminary results.
\subsection{Problem formulation }
Consider the following linear stochastic It\^o system with state-dependent noise:
\begin{align}
	\mathrm{d} x\left( t\right)& =\left(A x\left( t\right)+B u\left( t\right)+E v\left( t\right) \right) \mathrm{d} t+A_{1} x\left( t\right) \mathrm{d} W\left( t\right),\label{e1}\\
	z\left( t\right) &=\left[\begin{array}{c}
		C x\left( t\right)  \\
		D u\left( t\right) 
	\end{array}\right], \label{e2}
\end{align}
where $W(t)$ is assumed to be a one-dimensional
standard Brownian motion defined on the filtered probability
space $\left(\Omega, \mathcal{F},\left\{\mathcal{F}_{t}\right\}_{t \geq 0}, \mathbb{P}\right)$ with $\mathcal{F}_{t}=\sigma(W(s): 0 \leq s \leq t)$, $x\left( t\right)  \in \mathcal{R}^{n}$, $u\left( t\right)\in \mathcal{R}^{m}$, $v\left( t\right) \in \mathcal{R}^{p}$ and $z\left( t\right)\in \mathcal{R}^{q}$ are $\left\{\mathcal{F}_{t}\right\}$- adapted stochastic processes
representing the system state, the control input, the external disturbance and the controlled output, respectively.
 Assume that  the initial state $x(0)=x_{0}$ is  deterministic and that all coefficients $A$, $A_{1} $, $B$, $E$, $C$ and $D$ are  constant real matrices with  appropriate dimensions.    Denote $Q=C^{\top}C$ and $R=D^{\top}D$ in the sequel for brevity and assume throughout that $D$ has full column rank to ensure that $R$ is positive definite. 

For arbitrary $T$, $0<T<\infty$, when $\left(u, v, x_{0}\right) \in  \mathcal{L}_{\mathcal{F}}^{2}\left([0, T]; \mathcal{R}^{m}\right) \times \mathcal{L}_{\mathcal{F}}^{2}\left([0, T]; \mathcal{R}^{p}\right) \times \mathcal{R}^{n}$, there exists a unique strong solution $x(t)$ or, for
clarity, $x(t,u, v, x_{0}) \in \mathcal{L}_{\mathcal{F}}^{2}\left([0, T]; \mathcal{R}^{n}\right) $ to (\ref{e1}) with $x(0)=x_{0}$ \cite{krylov1995introduction}. We next give the formulation of stochastic $H_{\infty}$ control problem.
\begin{Problem}\rm \cite{chen2004stochastic}\par
	Find a  state feedback  control $u^{*}= -L^{*}x(t)  \in \mathcal{L}_{\mathcal{F}}^{2}\left(\mathcal{R}_{+}; \mathcal{R}^{m}\right)$, such that
	\begin{itemize} 
		\item[(\romannumeral1)] $u^{*}$ stabilizes  system~(\ref{e1}) internally, i.e,
		\begin{equation*}
			\lim _{\substack{{t \rightarrow \infty}}} \mathbb{E}\| x(t,u^{*},0,x_0)\|^2 =0.	
		\end{equation*}
		\item[(\romannumeral2)] $\|\mathscr{T}\|_{\infty} < \gamma$ for a given disturbance attenuation $\gamma>0$, where 
		\begin{equation*}
			\|\mathscr{T}\|_{\infty} 
			:= \sup _{\substack{v \in \mathcal{L}_{\mathcal{F}}^{2}\left(\mathcal{R}_{+};\mathcal{R}^{q} \right),
					v \neq 0}} \frac{\| \mathscr{T}\left( v\right) \|_{\mathcal{L}_{2}}}{\|v\|_{\mathcal{L}_{2}}}
		\end{equation*}
		with $\left( \mathscr{T}\left( v\right) \right) \left( t\right):=\left[\begin{array}{c}
			C x\left(t, u^{*}, v, 0\right) \\
			D u^{*} 
		\end{array}\right].$
	\end{itemize} 
\end{Problem}
If a control $u= -Lx(t)  \in \mathcal{L}_{\mathcal{F}}^{2}\left(\mathcal{R}_{+}; \mathcal{R}^{m}\right)$ only satisfies~(\romannumeral1), then this  admissible control $u(\cdot)$ is referred to as a (mean square) internally
stabilizing feedback control, $L$ is referred to as a internally
stabilization gain of system~(\ref{e1}).

\subsection{Preliminaries on stochastic $H_{\infty}$ control}
The following  lemmas and  proposition will be utilized frequently in this paper, which are listed below.  We begin by  providing the following  lemma, which establishes the criterion  for the asymptotic mean square  stability of a stochastic system.

\begin{lemma}\rm \label{lem1} \cite{zhang2004stabilizability}
	The system 
	\begin{equation}\label{lems}
			\mathrm{d} x\left( t\right) =A x\left( t\right) \mathrm{d} t+A_{1} x\left( t\right) \mathrm{d} W\left( t\right)
		\end{equation} 
	is asymptotically mean
	square stable if and only if $\sigma\left( \mathscr{L}_{A,A_{1}}\right) \subset \mathcal{C}_{-}$, where
	the generalized Lyapunov operator $\mathscr{L}_{A,A_{1}}$ is defined by
	\begin{equation*}
	 \mathscr{L}_{A,A_{1}}(X)=	X A+A^{\top} X+A_{1}^{\top} X A_{1}
		\end{equation*}
	and the spectral set of $\mathscr{L}_{A,A_{1}}$ is given by
	\begin{equation*}
		\sigma\left( \mathscr{L}_{A,A_{1}}\right):=\left\lbrace\lambda\Big| \mathscr{L}_{A,A_{1}}\left( X\right) =\lambda X, X \neq0,X \in \mathcal{S}^{n}\right\rbrace.
	\end{equation*}		
\end{lemma}
  In what follows, we give the  expression for the unique solution  to generalized Lyapunov equation (GLE):
 	\begin{equation}\label{gle}
 		\mathscr{L}_{A,A_{1}}\left( X\right)=X A+A^{\top} X+A_{1}^{\top} X A_{1}=Y, 
 	\end{equation}
  in Lemma~\ref{lem2} below. To this end, define an $\mathscr{H}$-representation matrix $H_{n}$ as
  \begin{equation*}
  	\begin{split}
  		H_{n} =& \left[  \operatorname{vec}(E_{11}), \cdots, \operatorname{vec}(E_{1n}), \operatorname{vec}(E_{22}), \cdots, \operatorname{vec}(E_{2n}),\right.\\
  		&\left.\cdots, \operatorname{vec}(E_{nn})\right],
  	\end{split} 
  \end{equation*}
  where $E_{ij} = (e_{lk})_{n \times n}$ with $e_{ij} = e_{ji} = 1$ and all other entries being zero. Then  define 
  \begin{equation*}	
  	\begin{split}
  		\mathcal{H}\left(A,A_{1},H_{n}\right):=&\left( H_{n}^{\top} H_{n}\right) ^{-1}H_{n}^{\top} \left( A  \otimes I_{n} +I_{n} \otimes A\right.\\
  		&\left.+A_{1}\otimes A_{1}\right)^{\top}H_{n}.
  	\end{split} 
  \end{equation*}
Based on  Lemma~\ref{lem1} above and Theorem~3.1 in  \cite{zhang2012cal}, which gives the criterion for the existence of a unique solution to GLE~(\ref{gle}),  we can  obtain the following
\begin{lemma}\rm \label{lem2}  
	If  system~(\ref{lems}) is asymptotically mean
square stable, then $\mathcal{H}\left(A,A_{1},H_{n}\right)$ is invertible 
	and for $Y \in \mathcal{S}^{n}$, the unique solution to GLE (\ref{gle}) is
	\begin{equation*}
		X=\left( \mathscr{L}_{A,A_{1}}\right)^{-1} \left( Y\right)=\operatorname{vecs}^{-1}\left( \mathcal{H}\left(A,A_{1},H_{n}\right) ^{-1}\operatorname{vecs}\left( Y\right) \right).
	\end{equation*} 
\end{lemma}
As can be directly verified, we have the following proposition, which is useful when taking the norm of a matrix in the sequel.
\begin{Proposition}\rm \label{pro:1} \cite{pang2021robust}
	For $M \in \mathcal{R}^{n \times n}$, $N \in \mathcal{R}^{n \times m}$, $\Delta M \in \mathcal{R}^{n \times n}$  and $\Delta N \in \mathcal{R}^{n \times m}$,  if $M$ and  $M+\Delta M$ are invertible,  then
	\begin{equation*}
		\begin{split}
			&\left\|M^{-1}N-\left( M+\Delta M\right)^{-1} \left( N+\Delta N\right)\right\|_{F} \leq \left\|M^{-1}\right\|_{F}\\
			&\left( \left\|\Delta N\right\|_{F}+ \left\|\left( M+\Delta M\right)^{-1}\right\|_{F}\left\|N+\Delta N\right\|_{F}\left\|\Delta M\right\|_{F}\right).
		\end{split}
	\end{equation*}	
\end{Proposition}
For stochastic $H_{\infty}$ control problem of system~(\ref{e1}), there is the following important theoretical result, which serves as the theoretical basis for the following work.
\begin{lemma}\rm  \label{llj} \cite{zhang2017stochastic}
	   Assume that  $(A, A_{1}|C)$ and $(A +\gamma^{-2} E E^{\top} P^{*}, A_{1} |C)$ are exactly detectable and that system~(\ref{e1}) is  internally stabilizable. Then $u^{*}(t) = -L^{*}x(t) = - R^{-1} B^{\top} P^{*} x(t)$ is an $H_{\infty}$ control and $v^{*}(t) = -F^{*}x(t) =\gamma^{-2}  E^{\top} P^{*}  x(t)$ is the corresponding worst-case disturbance, where $P^{*} \succeq 0$  is the unique stabilizing solution to generalized algebraic Riccati equation (GARE)
	 \begin{equation}\label{gare}
	 	\begin{split}
	 		&P A+A^{\top} P+A_{1}^{\top} P A_{1}-P B R^{-1} B^{\top} P\\
	 		&+\gamma^{-2}P E E^{\top} P+Q=0.
	 	\end{split}
	 \end{equation} 
\end{lemma}
A stabilizing solution in Lemma~\ref{llj} is defined as follows.
\begin{definition}\rm
	The solution $P \succeq 0$ to GARE~(\ref{gare}) is called a stabilizing solution, if the system
	\begin{equation*}		
	\mathrm{d} x\left( t\right) =\mathscr{A}(P) x\left( t\right)\mathrm{d}t +A_{1} x\left( t\right) \mathrm{d} W\left( t\right)			
	\end{equation*}
	is asymptotically mean square stable, where $\mathscr{A}(P):=A-B R^{-1} B^{\top} P+\gamma^{-2} E E^{\top} P$.
\end{definition}
\begin{Remark}\rm
	It is shown in \cite{dragan2006mathematical} that there exists a
	$\gamma^{*}$ such that for $\gamma< \gamma^{*}$, the stochastic $H_{\infty}$ control problem has no solution and  that 	when the conditions of Lemma~\ref{llj} are satisfied, if $\gamma\ge  \gamma^{*}\ge  0$, the GARE~(\ref{gare})  has
	a unique positive semi-definite solution.		
  Furthermore, as stated in Remark~9.2.1 of \cite{dragan2006mathematical}, $\gamma^{*}$ can be obtained by solving a semidefinite programming problem.	
\end{Remark}
It is vital to highlight that comprehensive knowledge of the system model is required to solve GARE~(\ref{gare}).
As a result, presenting algorithms that would converge to the solution of the stochastic $H_{\infty}$ control 
without the requirement of specific models of system dynamics is of special importance from the standpoint
of  control systems.
\section{Model-based SPU algorithm}
The model-based simultaneous policy update (SPU) technique for solving GARE~(\ref{gare})
is provided in this
section. We first provide Algorithm~\ref{alg:1}, and then demonstrate that the series formed by Algorithm~\ref{alg:1} is, in
fact, a Newton sequence, using Kantorovich's Theorem to demonstrate its convergence.
\subsection{Model-based SPU algorithm}
The SPU algorithm is a subtype of the policy iteration (PI) algorithm. PI is divided into two stages:
policy evaluation and policy improvement. The specified control policy and disturbance policy are assessed
using a scalar performance index in the policy evaluation stage. The performance index is then used to
produce new policies. The control policy and the disturbance policy are both improved at the same time
during the policy improvement stage. The procedure of model-based SPU algorithm is given in Algorithm~\ref{alg:1}.
\begin{algorithm}[htb]		
	\begin{algorithmic}[1]		
		\caption{Model-Based SPU Algorithm}
		\label{alg:1}
		\State   Choose an initial matrix $P^{0} \in \mathcal{P}_{0}$ ($\mathcal{P}_{0}$ is determined by Theorem~\ref{th2}) and a large enough number of iterations $N$. Set $L^{i} =R^{-1} B^{\top} P^{i}$,$
		F^{i} =-\gamma^{-2}  E^{\top} P^{i}$. Let the iteration index $i= 0$.
		\State (Policy Evaluation) Evaluate the performance of control policy $u^{i}\left( t\right):=-L^{i}x\left( t\right)$ and disturbance policy $v^{i}\left( t\right):=-F^{i}x\left( t\right)$ by solving the following GLE
		\begin{equation}\label{eq6}
			{P}^{i+1} \overline{A}^{i}+\left( \overline{A}^{i}\right) ^{\top} {P}^{i+1}+A_{1}^{\top} {P}^{i+1} A_{1}+\overline{Q}^{i}=0
		\end{equation}
		for $P^{i+1}$, where
		\begin{equation}\label{ai}
			\begin{split}
				\overline{A}^{i}:=&A-BL^{i}-EF^{i}\\
				=&A-B R^{-1} B^{\top} P^{i}+\gamma^{-2} E E^{\top} P^{i},
			\end{split}			
		\end{equation}
		and
		\begin{equation}\label{qi}
			\begin{split}
				\overline{Q}^{i}:=&Q+\left( L^{i}\right) ^{\top}RL^{i}-\gamma^{2} \left( F^{i}\right) ^{\top} F^{i}\\
				=&Q+P^{i} B R^{-1} B^{\top} P^{i}-\gamma^{-2} P^{i}E E^{\top} P^{i}.
			\end{split} 
			\end{equation}		
		\State (Policy Improvement) Update the control policy  and disturbance policy  simultaneously with the feedback gain matrices as follows:
		\begin{align}
			L^{i+1} &=R^{-1} B^{\top} P^{i+1},\label{lf11} \\
			F^{i+1} &=-\gamma^{-2}  E^{\top} P^{i+1}. \label{lf12}
		\end{align}	
		\State Set $ i \leftarrow i+1 $ and go back to Step 2 until $i=N-1$.
	\end{algorithmic}
\end{algorithm}
\subsection{Convergence analysis}
The convergence of Algorithm~\ref{alg:1} will be shown subsequently. To begin, we demonstrate that
the sequence  $\left\lbrace P^{i}\right\rbrace_{i=0}^{\infty}$  produced by Algorithm~\ref{alg:1} is inherently a Newton sequence. To that goal, consider
a Banach space $\mathcal{P} \subset \mathcal{S}^{n}_{+}$ supplied with Frobenius norm $\left\|\cdot\right\| _{F}$ and the mapping $\mathscr{F} : \mathcal{P}  \longrightarrow  \mathcal{P}$ that is
characterized as follows:
\begin{equation}\label{eq7}
	\begin{split}
		\mathscr{F}(P)=&	P A+A^{\top} P+A_{1}^{\top} P A_{1}-P B R^{-1} B^{\top} P\\
		&+\gamma^{-2}P E E^{\top} P+Q.
	\end{split}	
\end{equation}
The Fr\'{e}chet differential \cite{berger1977nonlinearity} of $\mathscr{F}$ at $P$ can thus be calculated as
\begin{equation}\label{eq8}
	\begin{split}
		\mathrm{d} \mathscr{F}(P;\Delta P )=&\mathscr{F}^{\prime}(P) \Delta P\\
		=&\Delta P \mathscr{A}(P)+\mathscr{A}(P)^{\top} \Delta P+A_{1}^{\top} \Delta PA_{1},
	\end{split}	
\end{equation}
where $\mathscr{F}^{\prime}(P)$ is the Fr\'{e}chet derivative of $\mathscr{F}$ at $P$, $\Delta P \in \mathcal{P}$, $\mathscr{A}(P):=A-B R^{-1} B^{\top} P+\gamma^{-2} E E^{\top} P$. Take into account mapping $\mathscr{N}(P) :=	P -\left( \mathscr{F}^{\prime}(P)\right)^{-1} \mathscr{F}(P)$. Construct a Newton iterative sequence $\left\lbrace P^{i}\right\rbrace_{i=0}^{\infty} $ as
\begin{equation}\label{eq10}
	\mathscr{N} (P^{i})=P^{i+1},i\in \mathcal{Z}_{+}.
\end{equation}
For the sequence generated by Algorithm~\ref{alg:1}, we have the following 
\begin{lemma}\rm \label{th1}
	The sequence $\left\lbrace P^{i}\right\rbrace_{i=0}^{\infty} $ generated by Algorithm~\ref{alg:1} and the Newton sequence
	(\ref{eq10}) are equivalent.
\end{lemma}
\begin{proof}\rm
	From (\ref{eq10}), it is evident that
	\begin{equation*}
	\mathscr{N} (P^{i})=P^{i+1}=P^{i} -\left( \mathscr{F}^{\prime}(P^{i})\right)^{-1} \mathscr{F}(P^{i}).
	\end{equation*}	
	Pre-multiply $ \mathscr{F}^{\prime}(P^{i})$ yields
	\begin{equation}\label{dairu}
	\mathscr{F}^{\prime}(P^{i}) P^{i+1}=\mathscr{F}^{\prime}(P^{i}) P^{i} - \mathscr{F}(P^{i}).
	\end{equation}	
	From  (\ref{eq8}), we have
	\begin{align}
	&\mathscr{F}^{\prime}(P^{i}) P^{i+1}={P}^{i+1} \overline{{A}}^{i}+\left( \overline{A}^{i}\right) ^{\top} {P}^{i+1}+A_{1}^{\top}{P}^{i+1} A_{1}, \label{11}	\\
	&\mathscr{F}^{\prime}(P^{i}) P^{i}={P}^{i} \overline{A}^{i}+\left( \overline{A}^{i}\right) ^{\top} {P}^{i}+A_{1}^{\top} {P}^{i} A_{1}.\label{12} 
	\end{align}	
	Substituting (\ref{11}) -- (\ref{12}) into (\ref{dairu}) and noting (\ref{eq7}), we can obtain
	\begin{equation*}	
	{P}^{i+1} \overline{A}^{i}+\left( \overline{A}^{i}\right)^{\top} {P}^{i+1}+A_{1}^{\top} {P}^{i+1}A_{1}=-\overline{Q}^{i}.
	\end{equation*}
	This completes the proof.\qed
\end{proof}
 As stated in Lemma~\ref{th1}, the iterative mechanism represented by (\ref{eq6}) is essentially a Newton's iteration. Unfortunately, the Newton's approach cannot guarantee
monotonic convergence on its own.   The following Kantorovich's Theorem ensures  the convergence   of Algorithm~\ref{alg:1} in Theorem~\ref{th2}. 
\begin{lemma}\rm  
	(Kantorovich's Theorem) \label{kth} \cite{tapia1971kantorovich} \par
Suppose that $\left\|\mathcal{F}^{\prime}\left(P^{1}\right)-\mathcal{F}^{\prime}
\left(P^{2}\right)\right\|_{F} \leq c_0 \left\|P^{1}-P^{2}\right\|_{F}$ for all $P^{1}, P^{2} \in \mathcal{P}$. If the following hypotheses are satisfied: \par
\begin{itemize} 
\item[(\romannumeral1)]	 $P^{0} \in \mathcal{P}_{1}:=\left\{P \Big|P \in \mathcal{P}~\text {such that}~   \mathcal{F}^{\prime}\left(P\right)^{-1} ~\text {exists}\right\} $;
\item[(\romannumeral2)] for constants $a,\ b$ such that	$\left\|\left(\mathcal{F}^{\prime}\left(P^{0}\right)\right)^{-1}\right\|_{F} \leq a$, $\left\|\left(\mathcal{F}^{\prime}\left(P^{0}\right)\right)^{-1} \mathcal{F}\left(P^{0}\right)\right\|_{F} \leq  b$, one has $c=c_0 a b \leq \frac{1}{2}$; 
\item[(\romannumeral3)]  $\mathcal{P}_{2}:=\left\{P \Big|\left\|P-P^{0}\right\|_{F} \leq \left(\frac{1-\sqrt{1-2 c}}{c}\right) b\right\} \subset \mathcal{P}_{1}$. 
\end{itemize} 
	Then the Newton iterative sequence $\left\lbrace P^{i}\right\rbrace_{i=0}^{\infty}$ exists and converges to $P^{*} \in \mathcal{P}_{2}$, resulting in $\mathscr{F}(P^{*})=0$. 
\end{lemma}	
 With the preparation  above, the convergence of Algorithm \ref{alg:1} is stated as follows. 
\begin{theorem}\rm \label{th2}
	  When the conditions of Lemma~\ref{llj} are satisfied,  the sequence generated by Algorithm~\ref{alg:1} $\left\lbrace P^{i}\right\rbrace_{i=0}^{\infty} $ starting from $P^{0} \in \mathcal{P}_{0}$ converges to  $P^{*}$, where $P^{*} \succeq 0$ is the stabilizing solution of GARE~(\ref{gare}).
\end{theorem}
\begin{proof}\rm
	Lemma~\ref{th1} demonstrates that the sequence $\left\lbrace P^{i}\right\rbrace_{i=0}^{\infty}$ generated by Algorithm~\ref{alg:1} is equivalent to the Newton sequence obtained
	by (\ref{eq10}).     When the conditions of Lemma~\ref{llj} are satisfied, GARE~(\ref{gare}) has a stabilizing solution  $P^{*} \succeq 0$  \cite{zhang2017stochastic}. Moreover, it is easy to show that the Fr\'{e}chet derivative  $\mathscr{F}^{\prime}(P)$ is Lipschitz continuous with constant $c_0$ in $\mathcal{P}$ and that there exists a constant $a_{0}$ such that $\left\|\left(\mathscr{F}^{\prime}\left(P^{*}\right)\right)^{-1}\right\|_{F} \leq a_{0}$, then  the theorem in \cite{rall1974note}
	implies  that $P^{0} \in \mathcal{P}_{0}:=\left\{P \Big| \|P-P^{*}\|_{F} \leq \frac{2-\sqrt{2}}{2 c_0a_{0} }\right\}$ can  guarantee the boundedness hypotheses of Kantorovich's Theorem. As a result of Kantorovich's Theorem, we can conclude  that the
	sequence $\left\lbrace P^{i}\right\rbrace_{i=0}^{\infty}$ is convergent,  which means that $\textstyle \lim_{i \to \infty}P^{i}=P^{*}$.\qed	 
\end{proof}
In  Algorithm~\ref{alg:1}, an appropriate $P^{0} \in \mathcal{P}_{0}$ can be chosen such that $L^{0} =R^{-1} B^{\top} P^{0}$ is an internally stabilizing gain. The following Theorem~\ref{th3} guarantees that such a $P^{0}$ always exists  and demonstrates that,
under certain conditions, all policies  $\left\lbrace u^{i}\right\rbrace_{i=0}^{\infty} $ with feedback gain matrices updated by (\ref{lf11}) are internally
stabilizing. Then by Lemma~\ref{lem2} and Eq.~(\ref{eq6}), we can conclude that the sequence
$\left\lbrace P^{i}\right\rbrace_{i=0}^{\infty} $ generated by Algorithm~\ref{alg:1} satisfies
\begin{equation}\label{pp1}
\begin{split}
P^{i+1}=&\left( \mathscr{L}_{\overline{A}^{i},A_{1}}\right)^{-1} \left( -\overline{Q}^{i}\right)\\ 
=&\operatorname{vecs}^{-1}\left(\mathcal{H}\left(\overline{A}^{i},A_{1},H_{n}\right)^{-1}\operatorname{vecs}\left( -\overline{Q}^{i}\right)\right).
\end{split}		
\end{equation}
If $P^i$ is viewed as the state and the iteration index $i$ is
viewed as the time, then 
	(\ref{pp1}) is a discrete-time nonlinear dynamical system and $P^*$ is an equilibrium by Theorem~\ref{th2}. The following Theorem~\ref{th3} states that $P^*$ is actually a locally exponentially stable equilibrium.
Furthermore,   Theorem~\ref{th3}  demonstrates that 
 Algorithm~\ref{alg:1} has a local quadratic convergence rate.  Its proof can be found in  Appendix A. 
\begin{theorem}\rm  \label{th3}  
	For any $\varepsilon_0 \in (0, 1)$, there exist $\delta_0(\varepsilon_0) > 0$ and $k_0(\delta_0)> 0$ 
	such that for any $P^{i} \in \mathcal{B}_{\delta_0}\left(P^{*}\right):=\left\lbrace P \in \mathcal{S}^{n}_{+}\Big| \|P-P^{*}\|_{F} \leqslant \delta_0\right\rbrace, i \in \mathcal{Z}_{+}$,
	\begin{itemize} 
		\item[(\romannumeral1)] $\sigma\left( \mathscr{L}_{\overline{A}^{i},A_{1}}\right) \subset \mathcal{C}_{-}$.
		\item[(\romannumeral2)] Algorithm~\ref{alg:1} has a local quadratic convergence rate, i.e.,
		\begin{equation}\label{th22}	
			\|P^{i+1 }-P^{*}\|_{F} \leq k_0\|P^{i}-P^{*}\|_{F} ^{2}.
		\end{equation}
		Especially,
		\begin{equation}\label{th23}
			\|P^{i+1 }-P^{*}\|_{F} \leq \varepsilon_0 \|P^{i}-P^{*}\|_{F}.
		\end{equation}
	\end{itemize} 	
\end{theorem}
\begin{Remark}\rm
	It is worth noting that the presented Algorithm~\ref{alg:1} is inherently Newton's method, which is not a
	global method. This indicates that Algorithm~\ref{alg:1} could not work when the initial matrix selected is distant
	from the solution of GARE~(\ref{gare}). Selecting appropriate initializations or establishing global methods is still
	a challenging problem up to now.
\end{Remark}
\begin{Remark}\rm
	To derive the stochastic $H_{\infty}$ control, Algorithm~\ref{alg:1} requires the complete information of the system dynamics.
	All input data for Problem 1 must be known at the start of the algorithm, and the results must  output immediately after solving  Problem
	1. As a result, Algorithm~\ref{alg:1} is a model-based off-line method.
\end{Remark}
\begin{Remark}\rm
	In Algorithm~\ref{alg:1}, the disturbance input must be adjusted in the prescribed manner (\ref{lf12}). In practice, however,  the disturbance
	and the state are independent, moreover, the disturbance cannot be specified. This issue will be addressed in
	the next section.
\end{Remark}
\section{Model-free off-policy RL algorithm}
In this section, we utilize the idea of RL approaches to present an SPU-based off-policy RL algorithm to
learn the solution of GARE~(\ref{gare}) without the need for system dynamics information. 

Assume that $u\left( t\right) $ and $v\left( t\right)$ are  the behavior policies that are implemented in system~(\ref{e1}) to generate data. On the contrary, $u^{i}\left( t\right)=-L^{i}x\left( t\right)$, $v^{i}\left( t\right)=-F^{i}x\left( t\right)$, $i \ge 1$ are the target policies that are being trained and improved in the $i$th iteration.
 
Rewrite the original system~(\ref{e1}) in the following form:
\begin{equation*}
	\begin{split}
		\mathrm{d} x\left( t\right) =&\left\lbrace \left( A-BL^{i}-EF^{i}\right)  x\left( t\right)+B\left[ L^{i} x\left( t\right)+u\left( t\right)\right] \right.\\
		&\left.+E\left[ F^{i} x\left( t\right)+v\left( t\right)\right]\right\rbrace  \mathrm{d} t+A_{1} x\left( t\right) \mathrm{d} W\left( t\right).
	\end{split} 
\end{equation*}
Let $P^{i+1}$ be the solution of the GLE (\ref{eq6}), and then by application of  It\^o's  formula to $x\left( \tau\right)^{\top} P^{i+1} x\left( \tau\right)$,  it yields that
\begin{equation}\label{eq23}	
\begin{array}{l} 
\mathrm{d}\left[x\left( \tau\right)^{\top} P^{i+1} x\left( \tau\right)\right] \\
=\left\lbrace  x^{\top}\left( \tau\right) \left[ P^{i+1}\overline{A}^{i}+\left(\overline{A}^{i}\right)^{\top} P^{i+1} + A_{1} ^{\top} P^{i+1}  A_{1}\right] x\left( \tau\right) \right. \\
\left.	+2\left(L^{i} x\left( \tau\right)+u\left( \tau\right)\right)^{\top} B^{\top} P^{i+1} x\left( \tau\right) \right.\\
\left.+2\left(F^{i} x\left( \tau\right)+v\left( \tau\right)\right)^{\top} E^{\top} P^{i+1}x\left( \tau\right)
\right\rbrace  \mathrm{d} \tau \\
+\left[  x\left( \tau\right)^{\top} \left(  P^{i+1}A_{1}+A_{1}^ {T} P^{i+1}\right) x\left( \tau\right)\right]  \mathrm{d}W\left( \tau\right) .
\end{array}	
\end{equation}
Integrating both sides of the equation (\ref{eq23}) throughout the trajectory
of (\ref{e1}) and noting (\ref{eq6}) and (\ref{qi}), we can obtain
\begin{equation}\label{eqin}
\begin{aligned}
&\left.\left(x\left( \tau\right)^{\top} P^{i+1} x\left( \tau\right)\right)\right|_{t} ^{t+\Delta t}\\ 
=& \int_{t}^{t+\Delta t} \left[  x\left( \tau\right)^{\top} \left(   \gamma^{2} \left( F^{i}\right) ^{\top} F^{i}-Q-\left( L^{i}\right) ^{\top}RL^{i} \right)  x\left( \tau\right)\right. \\
&+2\left(L^{i} x\left( \tau\right)+u\left( \tau\right)\right)^{\top}  B^{\top} P^{i+1} x\left( \tau\right)\\
& +\left.2 \left(F^{i} x\left( \tau\right)+v\left( \tau\right)\right)^{\top}  E^{\top} P^{i+1}x\left( \tau\right)\right]   \mathrm{d} \tau \\
&+\int_{t}^{t+\Delta t} \left[  x\left( \tau\right)^{\top} \left(  P^{i+1}A_{1}+A_{1}^ {T} P^{i+1}\right) x\left( \tau\right)\right]  \mathrm{d}W\left( \tau\right). 
\end{aligned}	
\end{equation}
Taking the conditional expectation on both sides of the equation (\ref{eqin}), one gets
\begin{equation}\label{a2}
\begin{array}{l} 
\mathbb{E}\left[x(t+\Delta t)^{\top} P^{i+1} x(t+\Delta t)\big| \mathcal{F}_{t}\right]-x(t)^{\top} P^{i+1} x(t)	\\
=\mathbb{E}\left\lbrace \int_{t}^{t+\Delta t}  \left[x\left( \tau\right)^{\top}   \left( \gamma^{2} \left( F^{i}\right) ^{\top} F^{i}-Q-\left( L^{i}\right) ^{\top}RL^{i}\right)    x\left( \tau\right)  \right. \right.\\
+\left.\left. 2\left(L^{i} x\left( \tau\right)+u\left( \tau\right)\right)^{\top}  B^{\top} P^{i+1} x\left( \tau\right)\right.\right.\\ 
+\left.\left.2  \left(F^{i} x\left( \tau\right)+v\left( \tau\right)\right)^{\top}  E^{\top} P^{i+1}x\left( \tau\right)\right]  \mathrm{d} \tau \Big| \mathcal{F}_{t}\right\rbrace.
\end{array}			
\end{equation}
 We now employ the crucial Eq.~(\ref{a2})
to solve the unknown vector $P^{i+1}$ in the least-squares sense. Construct the regression vector as 
\begin{equation*}
	\Xi^{i+1}:=\left[\operatorname{vecs}\left(\Theta^{i+1}_{1}\right)^{\top},\operatorname{vec}\left(\Theta^{i+1}_{2}\right)^{\top}, \operatorname{vec}\left(\Theta^{i+1}_{3}\right)^{\top}	\right]^{\top},
\end{equation*}
where
\begin{equation} \label{canshu}
	\Theta^{i+1}_{1}=P^{i+1}, \Theta^{i+1}_{2}=B^{\top} P^{i+1}, \Theta^{i+1}_{3}=E^{\top} P^{i+1}.
\end{equation}
Since there are  $\frac{n × (n + 1)}{2} +nm +np$ 
unknown components in the regression vector, we need to record the state along
trajectories at   $s$ intervals $\left( s\ge \frac{n × (n + 1)}{2} +nm +np\right) $: $\left[ t_{j}, t_{j} + \Delta t\right]$, $j=1,2,\ldots,s$,  where $0\leq t_{1}<t_{1} + \Delta t\leq t_{2}<\cdots <  t_{s}+\Delta t<\infty$. In other words, for initial
state $x\left( t_{j}\right) $ with $j=1,2,\ldots,s$ at each iteration, one needs to solve a set of equations
as
\begin{equation}\label{eq41}
\begin{array}{l} 
\mathbb{E}\left[x(t_{j}+\Delta t)^{\top} P^{i+1} x(t_{j}+\Delta t)\big| \mathcal{F}_{t_{j}}\right]-x(t_{j})^{\top} P^{i+1} x(t_{j})	\\
=\mathbb{E}\left\lbrace \int_{t_{j}}^{t_{j}+\Delta t}  \left[x\left( \tau\right)^{\top}   \left( \gamma^{2} \left( F^{i}\right) ^{\top} F^{i}-Q-\left( L^{i}\right) ^{\top}RL^{i}\right)    x\left( \tau\right)  \right. \right.\\
+\left.\left. 2\left(L^{i} x\left( \tau\right)+u\left( \tau\right)\right)^{\top}  B^{\top} P^{i+1} x\left( \tau\right)\right.\right.\\ 
+\left.\left.2  \left(F^{i} x\left( \tau\right)+v\left( \tau\right)\right)^{\top}  E^{\top} P^{i+1}x\left( \tau\right)\right]  \mathrm{d} \tau \Big| \mathcal{F}_{t_{j}}\right\rbrace.
\end{array}			
\end{equation}
Utilizing the collected data to define the data matrices $\Phi^{i}$ and $\Upsilon^{i}$ as
\begin{equation}\label{eq66}
\begin{array}{r} 
\Phi^{i}=\left[\Delta _{\tilde{x}},-2 \mathcal{I}_{x x}\left( I_{n}\otimes \left( L^{i}\right) ^{\top} \right)-2\mathcal{I}_{x u},\right. \\ 
\left.  
-2 \mathcal{I}_{x x}\left( I_{n}\otimes \left( F^{i}\right) ^{\top} \right)-2\mathcal{I}_{x v} \right],
\end{array}	
\end{equation}
and
\begin{equation*}
\Upsilon^{i}=\mathcal{I}_{\tilde{x}}\operatorname{vecs}\left(\gamma^{2} \left( F^{i}\right) ^{\top} F^{i}-Q-\left( L^{i}\right) ^{\top}RL^{i}\right),
\end{equation*}
where
\begin{equation*}
\begin{aligned}		
\Delta _{\tilde{x}} =&\left[ \delta_{1},\delta_{2},\ldots,\delta_{s}\right]^{\top},\\
\mathcal{I}_{\tilde{x}}=&\left[\tilde{\mathcal{X}}_{1},\tilde{\mathcal{X}}_{2}, \ldots, \tilde{\mathcal{X}}_{s}\right]^{\top}, \\
\mathcal{I}_{x x}=&\left[\mathcal{X}_{1}, \mathcal{X}_{2}, \ldots, \mathcal{X}_{s}\right]^{\top},  \\
\mathcal{I}_{x u}=&\left[\mathcal{U}_{1}, \mathcal{U}_{2}, \ldots, \mathcal{U}_{s}\right]^{\top}, \\
\mathcal{I}_{x v}=&\left[\mathcal{V}_{1}, \mathcal{V}_{2}, \ldots, \mathcal{V}_{s}\right]^{\top} \\
\end{aligned}
\end{equation*}
with
\begin{equation*}
\begin{aligned}		
\delta_{j}&:=\mathbb{E}\left[\tilde{x}\left( t_{j}+\Delta t\right)   \Big|\mathcal{F}_{t_{j}}\right] -\tilde{x}\left( t_{j}\right),\\
\tilde{\mathcal{X}}_{j}&:=\mathbb{E} \left[\int_{t_{j}}^{t_{j}                                  +\Delta t}  \tilde{x}\left( \tau\right)    \mathrm{d} \tau\Big|\mathcal{F}_{t_{j}}\right], \\
\mathcal{X}_{j}&:=\mathbb{E} \left[\int_{t_{j}}^{t_{j}                                  +\Delta t}  x\left( \tau\right) \otimes  x\left( \tau\right)    \mathrm{d} \tau\Big|\mathcal{F}_{t_{j}}\right], \\
\mathcal{U}_{j}&:=\mathbb{E} \left[\int_{t_{j}}^{t_{j}                                  +\Delta t}  x\left( \tau\right) \otimes  u\left( \tau\right)    \mathrm{d} \tau\Big|\mathcal{F}_{t_{j}}\right], \\
\mathcal{V}_{j}&:=\mathbb{E} \left[\int_{t_{j}}^{t_{j}                                  +\Delta t}  x\left( \tau\right) \otimes  v\left( \tau\right)    \mathrm{d} \tau\Big|\mathcal{F}_{t_{j}}\right].\\
\end{aligned}
\end{equation*}
According to Kronecker product representation and based on Kronecker product property, the set of equations (\ref{eq41}) can be rewritten as
\begin{equation}\label{lseq}
\Phi^{i} \Xi^{i+1}=\Upsilon^{i}.
\end{equation}
Then, under the assumption that $\Phi^{i}$  has full column rank, which may be assured by some rank condition in the following Lemma~\ref{lem6},
the least-square solution of equation (\ref{lseq}) is given by
\begin{equation}\label{lss}
\Xi^{i+1}=\left(\left(\Phi^{i}\right)^{\top} \Phi^{i}\right)^{-1}\left(\Phi^{i}\right)^{\top} \Upsilon^{i}.
\end{equation} 
\begin{lemma}\rm \label{lem6} 
	Assume that there exists a positive integer $S$, such that for all $s\geqslant S$,
	\begin{equation}\label{rk}
		\operatorname{rank}\left(\left[\mathcal{I}_{\tilde{x}},\mathcal{I}_{x u},\mathcal{I}_{x v}\right]\right)=\frac{n(n+1)}{2}+ nm+np,
	\end{equation}
	then $\Phi^{i}$ has full column rank for all $i \in \mathcal{Z}_{+}$.
\end{lemma}
\begin{proof}\rm
	To prove this lemma, we  only need to show that for each given  $i \in \mathcal{Z}_{+}$, $\Phi^{i}\Lambda=0$ 	
	has unique solution $\Lambda=0$.
	Define $\Lambda	=\left[\left(\operatorname{vecs}\left( X\right) \right) ^{\top},\left(\operatorname{vec}\left(Y\right) \right) ^{\top},\left( \operatorname{vec}\left(Z\right)\right) ^{\top}	\right]^{\top}$,  where  $X\in \mathcal{S}^{n},Y \in \mathcal{R}^{m \times n} $ and  $Z\in \mathcal{R}^{p \times n}$.
	
	By  equation (\ref{eq23}) and the definition of $\Phi^{i}$ in (\ref{eq66}), we have
	\begin{equation}\label{eq28}
	\Phi^{i} \Lambda=\left[\mathcal{I}_{\tilde{x}},2\mathcal{I}_{x u},2\mathcal{I}_{x v}\right] \left[\begin{array}{c}
	\operatorname{vecs}\left( \Lambda_{1}\right)\\
	\operatorname{vec}\left(\Lambda_{2}\right)\\
	\operatorname{vec}\left(\Lambda_{3}\right)
	\end{array}\right]=0,
	\end{equation}	
	where	
	\begin{align}
	\Lambda_{1}=&X \overline{A}^{i}+\left( \overline{A}^{i}\right) ^{\top} X+A_{1}^{\top} X A_{1}\notag\\
	&+\left( L^{i}\right) ^{\top}\left( B^{\top}X-Y\right) +\left( B^{\top}X-Y\right)^{\top}L^{i}\notag \\
	&+\left( F^{i}\right) ^{\top}\left( E^{\top}X-Z\right) +\left( E^{\top}X-Z\right)^{\top}F^{i},\label{eq291}\\
	\Lambda_{2}=&B^{\top}X-Y,\label{eq292}\\ 
	\Lambda_{3}=&E^{\top}X-Z.\label{eq293}
	\end{align}	
	Under the rank condition in (\ref{rk}), we get that  the only solution to (\ref{eq28}) is  $\Lambda_{1} = 0$, $\Lambda_{2} = 0$ and
	$\Lambda_{3} = 0$.
	
	In accordance with (\ref{eq292}) and (\ref{eq293}), we have that $Y= B^{\top}X$, $Z =E^{\top}X$ and that (\ref{eq291}) is reduced to the following GLE:
	\begin{equation}\label{eq30}
	X \overline{A}^{i}+\left( \overline{A}^{i}\right) ^{\top} X+A_{1}^{\top} X A_{1}=0.
	\end{equation}
	Because $\sigma\left( \mathscr{L}_{\overline{A}^{i},A_{1}}\right)$ does not contain zero eigenvalues, as demonstrated in the proof of Theorem~\ref{th3}, the only
	solution to (\ref{eq30}) is $X = 0$. At last,  (\ref{eq292}) and (\ref{eq293}) provide  $Y = 0$ and $Z = 0$. To summarize, we have $\Lambda = 0$.
	As a result, $\Phi^{i}$ must have full column rank for all $i \in \mathcal{Z}_{+}$. The proof is complete.	
\end{proof}
\begin{Remark}\rm
	The rank condition in Lemma~\ref{lem6} is analogous to the persistent excitation (PE) requirement \cite{willems2005note}
	in some sense. In other words, both the rank condition in Lemma~\ref{lem6} and the PE condition are intended to
	have a unique solution to (\ref{lseq}). In practice, we frequently inject the exploration noise into the input to do
	this.
\end{Remark}
We can now present the  model-free off-policy RL algorithm.
\begin{algorithm}[htb]
	\caption{Model-Free Off-Policy RL Algorithm}
	\label{alg:2}
	\begin{algorithmic}[1]
		\State  Choose an initial matrix $P^{0} \in \mathcal{P}_{0}$. Set $L^{0} =R^{-1} B^{\top} P^{0}$,$
		F^{0} =-\gamma^{-2}  E^{\top} P^{0}$. Apply admissible control policies $u\left( t\right) =-Lx\left( t\right)$, $v\left( t\right)=-Fx\left( t\right)$ with exploration noises to system~(\ref{e1}) and collect the input and state data. 	
		\State  Calculate $\mathcal{I}_{\tilde{x}},\mathcal{I}_{x u},\mathcal{I}_{x v}$ until the rank condition in (\ref{rk}) is satisfied.
		\State  Select a large enough number of iterations $N$. Set $u^{i}\left( t\right)=u\left( t\right)$, $v^{i}\left( t\right)=v\left( t\right)$. Let the iteration index $i= 0$. 
		\For {$i=0,1,\cdots, N-1$}
		\State Construct the data matrices $\Phi^{i}$ and $\Upsilon^{i}$. Solve  equation (\ref{lseq}) for $\Xi^{i+1}$.	
		\State  Update policies $u^{i}\left( t\right)=-L^{i}\left( t\right)$ and $v^{i}\left( t\right)=-F^{i}\left( t\right)$ simultaneously with the iterative feedback gain matrices as	
		\begin{align}
		L^{i+1} &=R^{-1} \Theta^{i+1}_{2},\label{lf31}\\
		F^{i+1} &=-\gamma^{-2}  \Theta^{i+1}_{3}.\label{lf32}
		\end{align}	
		\EndFor \\
		\Return  $P^{N}$, $L^{N}$ and $F^{N}$.
	\end{algorithmic}
\end{algorithm}	
\begin{Remark}\rm
	Algorithm~\ref{alg:2} may be separated into two phases. Lines 1--2 of Algorithm~\ref{alg:2} constitute the data
	collecting phase, Lines 3--7 constitute the learning phase. In the process of learning, Algorithm~\ref{alg:2} requires no prior knowledge of system dynamics. In addition, the target policies are unrelated to the behavior policies in Algorithm ~\ref{alg:2}, hence Algorithm~\ref{alg:2} is a model-free off-policy algorithm. Furthermore,  the disturbance policy which is specified and
	updated in (\ref{lf32}) does not need to be applied to the system.
\end{Remark}

In Theorem~\ref{th4}, we then show how the suggested Algorithm~\ref{alg:2} converges.
\begin{theorem}\rm \label{th4} 
	 If  the conditions of Lemma~\ref{llj} and the rank condition in Lemma~\ref{lem6} are satisfied, then the sequence generated by Algorithm~\ref{alg:2} $\left\lbrace P^{i}\right\rbrace_{i=0}^{\infty} $ starting from $P^{0} \in \mathcal{P}_{0}$ converges to  $P^{*}$, where $P^{*} \succeq 0$ is the stabilizing solution of GARE~(\ref{gare}).	
\end{theorem}
\begin{proof}\rm
	Suppose that 
	\begin{equation*}	
	\Gamma:=\left[\left(\operatorname{vecs}\left(\hat{X}\right) \right) ^{\top},\left(\operatorname{vec}\left(\hat{Y}\right) \right) ^{\top},\left( \operatorname{vec}\left(\hat{ Z}\right)\right) ^{\top}	\right]^{\top}
	\end{equation*}
	satisfies 
	\begin{equation}\label{eq52}
	\Phi^{i} \Lambda-\Upsilon^{i}=0,
	\end{equation}
	where  $\hat X\in \mathcal{S}^{n}, \hat{Y} \in \mathcal{R}^{m \times n} $ 
	and $\hat{ Z}\in \mathcal{R}^{p \times n}$.
	
	By definitions of $\Phi^{i}$ and $\Upsilon^{i}$, equation (\ref{eq52}) is equivalent to
	\begin{equation}\label{eq53}
	\left[\mathcal{I}_{\tilde{x}},2\mathcal{I}_{x u},2\mathcal{I}_{x v}\right] \left[\begin{array}{c}
	\operatorname{vecs}\left( \Gamma_{1}\right)\\
	\operatorname{vec}\left(\Gamma_{2}\right)\\
	\operatorname{vec}\left(\Gamma_{3}\right)
	\end{array}\right]=0,
	\end{equation}
	where
	\begin{align}
	\Gamma_{1}=&\hat{X} \overline{A}^{i}+\left( \overline{A}^{i}\right) ^{\top} \hat{X}+A_{1}^{\top} \hat{X} A_{1}\notag\\
	&+\left( L^{i}\right) ^{\top}\left( B^{\top}\hat{X}-\hat{Y}\right) +\left( B^{\top}\hat{X}-\hat{Y}\right)^{\top}L^{i}\notag \\
	&+\left( F^{i}\right) ^{\top}\left( E^{\top}\hat{X}-\hat{ Z}\right) +\left( E^{\top}\hat{X}-\hat{ Z}\right)^{\top}F^{i}\notag\\
	&-\gamma^{2} \left( F^{i}\right) ^{\top} F^{i}+Q+\left( L^{i}\right) ^{\top}RL^{i},\label{eq541} \\
	\Gamma_{2}=&B^{\top}\hat{X}-\hat{Y},\label{eq542}\\
	\Gamma_{3}=&E^{\top}\hat{X}-\hat{ Z}.\label{eq543}
	\end{align}
	Under the rank condition in (\ref{rk}), we know that the only solution to (\ref{eq53}) is $\Gamma_{1} = 0$, $\Gamma_{2} = 0$ and
	$\Gamma_{3} = 0$. Substituting $\hat{Y}=B^{\top}\hat{X},\hat{ Z} =E^{\top}\hat{X}$
	into (\ref{eq541}), we obtain
	\begin{equation}\label{eq56}
	\begin{array}{r} 
	\hat{X} \overline{A}^{i}+\left( \overline{A}^{i}\right) ^{\top} \hat{X}+A_{1}^{\top} \hat{X} A_{1}+Q\\
	+\left( L^{i}\right) ^{\top}RL^{i}-\gamma^{2} \left( F^{i}\right) ^{\top} F^{i}=0.
	\end{array}		
	\end{equation}
	On the other hand, according to Lemma~\ref{lem6},  $\Phi^{i}$ has full column rank
	for all $i \in \mathcal{Z}_{+}$, hence the solution to (\ref{eq52}) is unique. So $\Gamma=\Xi^{i+1}$, that is, $\hat{X}=P^{i+1}$, $\hat{Y}=B^{\top}P^{i+1}$ and $\hat{ Z}=E^{\top}P^{i+1}$. Plugging these three equations into (\ref{eq56}) and noting (\ref{qi}), we have
	\begin{equation*}
	\begin{array}{r} 
	P^{i+1} \overline{A}^{i}+\left( \overline{A}^{i}\right) ^{\top} P^{i+1}+A_{1}^{\top} P^{i+1} A_{1}+Q\\+P^{i} B R^{-1} B^{\top} P^{i}
	-\gamma^{-2} P^{i}E E^{\top} P^{i}=0,
	\end{array}		
	\end{equation*} 
	which is exactly GLE (\ref{eq6}). Noting (\ref{canshu}), PI by (\ref{lseq}),  (\ref{lf31}) and (\ref{lf32}) is equivalent to (\ref{eq6}),  (\ref{lf11}) and (\ref{lf12}),  respectively. According to Theorem~\ref{th2}, the convergence is proved.\qed
\end{proof}
	It is necessary to note that  the solution obtained using Eq.~(\ref{lseq}) is not a true solution in general, but rather a least-squares estimation using the collected
	state and input data.  Because of the presence of stochastic noises, the unknown  stochastic noises will distort the state trajectories in an unpredictable way.  Furthermore,  the conditional expectations in data matrices $\Phi^{i}$ and $\Upsilon^{i}$ cannot be obtained
	exactly. In practice, we adopt numerical averages to approximate the conditional expectations and use summations to approximate the integrals, the corresponding approximations obtained are then distinguished from the original notation by adding the superscript $\hat{}$ to each of them. More specifically, if we have $\bar L$
	sample paths  $x^{(l)},l=1,2,\cdots,\bar L$  with the data collected at time $t_{j_{k}}, k =
	1, 2,\cdots, \bar K$, where  $t_{j}= t_{j_{0}}< t_{j_{1}}<\cdots < t_{j_{\bar K}}= t_{j}+\Delta t$, we approximate  $\mathbb{E}\left[\tilde{x}\left( t_{j}+\Delta t\right)  ^{\top} \Big|\mathcal{F}_{t_{j}}\right]$ in $\Delta _{\tilde{x}}$  and $
	\mathbb{E} \left[\int_{t_{j}}^{t_{j}                                  +\Delta t}  x\left( \tau\right)^{\top} \otimes  x\left( \tau\right)^{\top}    \mathrm{d} \tau\Big|\mathcal{F}_{t_{j}}\right]$ in $\mathcal{I}_{x x}$ respectively as follows
	 \begin{equation*}
	 \begin{aligned}
	 &\mathbb{E}\left[\tilde{x}\left( t_{j}+\Delta t\right)  ^{\top} \Big|\mathcal{F}_{t_{j}}\right] \\ 
	 \approx  &\frac{1}{\bar L} \sum_{l=1}^{\bar L}\tilde{x}^{(l)}\left( t_{j}+\Delta t\right) ^{\top},\\	
			&\mathbb{E} \left[\int_{t_{j}}^{t_{j}                                  +\Delta t}  x\left( \tau\right)^{\top} \otimes  x\left( \tau\right)^{\top}    \mathrm{d} \tau\Big|\mathcal{F}_{t_{j}}\right]  \\
			\approx & \frac{1}{\bar L} \sum_{l=1}^{\bar L}\left[\sum_{k=1}^{\bar K}
			\left( 	x^{(l)}(t_{j_{k}})^{\top} \otimes x^{(l)}(t_{j_{k}})^{\top}\right) \times \left( t_{j_{k}}-t_{j_{k-1}}\right) 
			\right].
		\end{aligned}	
	\end{equation*}
	The integrals in $\mathcal{I}_{x u}$ and $\mathcal{I}_{x v}$ can be obtained in the same way as the approximate.

\begin{Remark}\rm \label{re1}
	In addition to the above mentioned  error arising from the calculation  of the  least-squares solutions from the input and state data, 
 the sample estimation error and approximation error arising from the computation of conditional expectations and integrals,
	  the sources of  error also include, for example, the residual induced by an earlier termination of the iteration to numerically solve  GLE  and so on. 
\end{Remark}

  If the  the effects of  errors are not taken into account, we can obtain the convergence of Algorithm~\ref{alg:2} by proving that Algorithm~\ref{alg:2} is equivalent to Algorithm~\ref{alg:1}, see Theorem~\ref{th4} for details. However,  error  is unavoidable. Taking into account the presence of errors,  we can now present the data-driven model-free off-policy RL algorithm. By executing Algorithm~\ref{alg:3}, we can obtain an estimate $\hat{P}^{N}$ of ${P}^{*}$.
  \begin{algorithm}[htb]
  	\caption{Data-Driven Model-Free Off-Policy RL Algorithm}
  	\label{alg:3}
  	\begin{algorithmic}[1]
  		\State  Choose an initial matrix $\hat{P}^{0} \in \mathcal{P}_{0}$. Set $\hat{L}^{0} =R^{-1} B^{\top} \hat{P}^{0}$, $
  		\hat{F}^{0} =-\gamma^{-2}  E^{\top} \hat{P}^{0}$. Apply  control policies $u\left( t\right) =-\hat{L}^{0}x\left( t\right)$, $v\left( t\right)=-\hat{F}^{0}x\left( t\right)$ with exploration noises to system~(\ref{e1}) and collect the input and state data. 	
  		\State  Calculate $\hat{\mathcal{I}}_{\tilde{x}},\hat{\mathcal{I}}_{x u},\hat{\mathcal{I}}_{x v}$ until the rank condition in Lemma~\ref{lem6} is satisfied.
  		\State  Select a large enough number of iterations $N$. Set $u^{i}\left( t\right)=-\hat{L}^{i}x\left( t\right)$, $v^{i}\left( t\right)=-\hat{F}^{i}x\left( t\right)$. Let the iteration index $i= 0$. 
  		\For {$i=0,1,\cdots, N-1$}
  		\State Construct the data matrices $\hat{\Phi}^{i}$ and $\hat{\Upsilon}^{i}$. Solve  the equation $\hat{\Phi}^{i} \hat{\Xi}^{i+1}=\hat{\Upsilon}^{i}$ for $\hat{\Xi}^{i+1}$.	
  		\State  Update policies $u^{i}\left( t\right)$ and $v^{i}\left( t\right)$ simultaneously with the iterative feedback gain matrices as	
  		\begin{align*}
  		\hat{L}^{i+1} &=R^{-1} \hat{\Theta}^{i+1}_{2},\\
  		\hat{F}^{i+1} &=-\gamma^{-2}  \hat{\Theta}^{i+1}_{3}.
  		\end{align*}	
  		\EndFor \\
  		\Return  estimates $\hat{P}^{N}$, $\hat{L}^{N}$ and $\hat{F}^{N}$.
  	\end{algorithmic}
  \end{algorithm}	
\section{Robustness analysis} 
In this section, considering the effects of  errors, the robustness
of RL to errors in the learning process is  studied. We may deduce that
Algorithm~\ref{alg:3} is robust to noisy data induced by modest unknown perturbations in the system dynamics
when the initial condition is in the neighborhood of the true solution by viewing the learning processes as
dynamical systems.

Define
	\begin{equation*}
	\begin{array}{l} 
		\mathscr{R}\left(Z,L,F\right)
		:=\left[I_n,-L^{\top},-F^{\top}\right]  Z\left[I_n,-L^{\top},-F^{\top}\right]^{\top}		
	\end{array}		
\end{equation*}	
for  $Z \in \mathcal{S}^{n+m+p},L \in \mathcal{R}^{m \times n},F \in \mathcal{R}^{p \times n}$ and 
\begin{equation*}
	M(P):=\left[\begin{array}{ccc}
		Q+A^{\top} P+P A+A_{1}^{\top} P A_{1} &P B&PE \\
		B^{\top} P & R&0\\
		E^{\top} P & 0&-\gamma^{2}I_p
	\end{array}\right] 
\end{equation*}
for  $P \in \mathcal{S}^{n}$.   Taking into account the  errors due to a number of factors mentioned in Remark \ref{re1}, we  propose the following
procedure in the context of unmeasurable stochastic noise.
\begin{algorithm}[htb]		
	\begin{algorithmic}[1]		
		\caption{Robust SPU}
		\label{alg:33}
		\State  Choose an initial matrix $\hat{P}^{0} \in \mathcal{P}_{0}$. Set $\hat{L}^{0} =R^{-1} B^{\top} \hat{P}^{0}$, $
			\hat{F}^{0} =-\gamma^{-2}  E^{\top} \hat{P}^{0}$.  Let the iteration index $i= 0$.
		\State (Inexact Policy evaluation) Obtain  $\hat{M}^{i+1}=\overline {M}^{i+1}+\Delta M^{i+1} \in \mathcal{S}^{n+m+p}$ (e.g., by
		approximately evaluating the performance of $\hat{L}^{i}$ and $\hat{F}^{i}$ directly from the input and state data, see Algorithm~\ref{alg:3}) as an approximation of  $\overline{M}^{i+1}$, where $\Delta M^{i+1}$ is an error, $\overline {M}^{i+1}:=M\left(\hat{P}^{i+1}\right) $, $\hat{P}^{i+1}$ is the solution of
		\begin{equation}\label{p21}
			\mathscr{R}\left(\overline{M}^{i+1}, \hat{L}^{i},\hat{F}^{i}\right)=0.
		\end{equation}		
		\State (Policy improvement) Construct new control and disturbance gains simultaneously by
		\begin{align}
			\hat{L}^{i+1}=&  \left[ \hat{M}^{i+1}\right] _{22} ^{-1}\left[ \hat{M}^{i+1}\right] _{21},\label{lf61}\\
			\hat{F}^{i+1}=& \left[  \hat{M}^{i+1}\right] _{33} ^{-1}\left[ \hat{M}^{i+1}\right] _{31},\label{lf62}	
		\end{align}		
		where $\left[  \hat{M}^{i+1}\right] _{ij}$ is the $(i,j)$th block of the  block matrix $\hat{M}^{i+1}$.	
		\State Set $ i \leftarrow i+1 $ and go back to Step 2.
	\end{algorithmic}
\end{algorithm}

In Algorithm~\ref{alg:33}, suppose $\hat L^{0} =R^{-1} B^{\top} \hat P^{0},
\hat F^{0} =-\gamma^{-2}  E^{\top} \hat P^{0}$ and $\Delta M^0 = 0$,
where $\hat P^{0}$ is chosen such that $\hat L^{0}$ is internally stabilizing. If $\hat{L}^{i}$ is internally stabilizing,  $\left[  \hat M^{i}\right]  _{22}$ and $\left[  \hat M^{i}\right]_{33}$ are invertible for all $i \in \mathcal{Z}_{+}$
(the above
assumptions may hold under certain conditions, as shown in item (\romannumeral1) in Theorem~\ref{th6}), then from (\ref{p21}), we know that the sequence $\left\lbrace \hat P^{i}\right\rbrace_{i=0}^{\infty} $
generated by Algorithm~\ref{alg:33} satisfies
\begin{equation}\label{hatp}
			\hat P^{i+1}=\left( \mathscr{L}_{\hat{\overline{A}}^{i},A_{1}}\right)^{-1} \left( -\hat{Q}^{i}_1\right)
		+ \mathcal{D}\left(\overline{M}^{i+1}, \Delta M^{i+1}\right),
\end{equation}
where 
\begin{equation*}
	\begin{array}{l} 
		\mathcal{D}\left(\overline{M}^{i+1}, \Delta M^{i+1}\right) \\
		=\left( \mathscr{L}_{A-B\hat L^{i}-E\hat F^{i},A_{1}}\right)^{-1} \left(- \hat{Q}_{2}^{i}\right)
		-\left( \mathscr{L}_{\hat{\overline{A}}^{i},A_{1}}\right)^{-1} \left(- \hat{Q}^{i}_1\right)
	\end{array}
\end{equation*}
with $\hat{\overline{A}^{i}}:= \mathscr{A}(\hat P^{i})$ and
\begin{equation*}
	\begin{split}			
		\hat{Q}^{i}_1:=&Q+\hat P^{i} B R^{-1} B^{\top} \hat P^{i}
		-\gamma^{-2} \hat P^{i}E E^{\top} \hat P^{i},\\
		\hat{Q}^{i}_2:=&Q+\left( \hat L^{i}\right) ^{\top}R\hat L^{i}-\gamma^{2} \left( \hat F^{i}\right) ^{\top} \hat F^{i}.\\
	\end{split} 	
\end{equation*}
Here, the dependence of $\mathcal{D}$ on $\overline{M}^{i+1}$ and  $\Delta M^{i+1}$ comes from (\ref{lf61}) and (\ref{lf62}). If $\hat{P}^{i}$ is viewed as the state, $\Delta M^{i}$ is viewed  as the disturbance input, then the next theorem is derived based on  Theorem~\ref{th3} and shows that discrete-time nonlinear dynamical
system~(\ref{hatp}) is locally input-to-state stable. The proof is
given in Appendix B.
\begin{theorem}\rm  \label{th6}  
	For $\varepsilon_0 $ and $\delta_0(\varepsilon_0)$ in Theorem~\ref{th3},  there exists a
	$\delta_1(\delta_0)>0$  such that if $\|\Delta M\|_{l_{\infty}}	:= \sup _{\substack{i \in \mathcal{Z}_{+} }} \|\Delta M^{i}\|_{F}$\\ $<\delta_1$, $\hat P^{0} \in \mathcal{B}_{\delta_0}\left(P^{*}\right)$, we have the following conclusions:
	\begin{itemize} 
		\item[(\romannumeral1)] $\sigma\left( \mathscr{L}_{A-B\hat L^{i}-E\hat F^{i},A_{1}}\right) \subset \mathcal{C}_{-}$, $\left[  \hat M^{i}\right]  _{22}$ and $\left[  \hat M^{i}\right]_{33}$ are invertible, for all $i \in \mathcal{Z}_{+}$.
		\item[(\romannumeral2)] The following local input-to-state stability holds:
		\begin{equation*}
			\|P^{i}-P^{*}\|_{F} \leq\alpha \left(\left\|\hat{P}^{0}-P^{*}\right\|_{F},i\right)+\beta \left( \|\Delta M\|_{l_{\infty}}\right),
		\end{equation*}
		where  $\alpha \left( s,i\right) =\varepsilon_{0}^{i}s \in \mathcal{KL}$, $\beta \left( s\right) =\frac{k_{3}}{1-\varepsilon_{0}}s \in \mathcal{K}$, $s\in \mathcal{R}$ and $k_{3}(\delta_0)>0$.
		\item[(\romannumeral3)] $\textstyle \lim_{i \to \infty}\|\Delta M^{i}\|_{F}=0 \Rightarrow   \textstyle \lim_{i \to \infty}\|\hat P^{i}-P^{*}\|_{F}=0$.
	\end{itemize} 
\end{theorem}

Evidently, Theorem~\ref{th6} states that in Algorithm~\ref{alg:33}, if $\hat P^0$ is close to $P^*$  and the error $\Delta M$ has a small  $l_{\infty}$-
norm, the cost of the produced policies is not greater than a constant proportionally to the error's $l_{\infty}$-norm. The smaller the error, the better the final policies developed. 
In other words,  Algorithm~\ref{alg:33} is not sensitive to small disturbances when the initial condition is in a neighbourhood
of the true solution. In terms of Algorithm~\ref{alg:3}, it is a specific method to construct estimation $\hat M^{i}$ in Algorithm~\ref{alg:33} directly from input
and state data. We may deduce from Theorem~\ref{th6} that Algorithm~\ref{alg:3}, a data-driven version of Algorithm~\ref{alg:33}, is robust
to multiplicative noise in the system dynamics.
\begin{Remark}\rm
	The proposed algorithms are shown to converge only for initial solution satisfying $P^{0} \in \mathcal{P}_{0}$ or $\hat{P}^{0} \in \mathcal{P}_{0}$, which may be unavailable in some cases. It should be noted that, in theory, the restriction $P^{0} \in \mathcal{P}_{0}$ and $\hat{P}^{0} \in \mathcal{P}_{0}$ can be removed and the proposed algorithms converge for all  stabilizing initial solutions, as in the convergence results obtained in \cite{pang2022reinforcement} for model-free algorithm posed to the stochastic optimal control problem,  and a similar process of deriving convergence conclusions for the stochastic  $H_{\infty}$ control  problem is part of our subsequent research.  In practice, we directly utilized this theoretical result by choosing the appropriate initial solution to make it a stabilizing solution. Since the system model is completely
	known for Algorithm~\ref{alg:1}, a practical method is also to select the matrix $P^{0}$ such that $\sigma\left( \mathcal{L}_{\overline{A}^{0},A_{1}}\right) \subset \mathcal{C}_{-}$. Thus, if the open-loop system (i.e., $u = 0$ and
	$v = 0$) is stable, we can simply select $P^{0}= 0$. For  Algorithm~\ref{alg:3}, since the system information is unknown, we can observe the  trend of state trajectories to select the appropriate $\hat{P}^{0}$.  Set $\hat{L}^{0} =R^{-1} B^{\top} \hat{P}^{0}$, $
	\hat{F}^{0} =-\gamma^{-2}  E^{\top} \hat{P}^{0}$. It is worth noting that this setup only presents the theoretical relationship between $\hat{L}^{0}$,$
	\hat{F}^{0}$ and $\hat{P}^{0}$, respectively; both $B$ and $E$ are unknown at the time of implementing Algorithm~\ref{alg:3}. We pick  $\hat{L}^{0}$ and $\hat{F}^{0}$ at random to apply to  system~(1)
	and observe the  trend of state trajectories. If there exist $\hat{L}^{0}$ and $\hat{F}^{0}$ such that the associated  state trajectories  go to a neighborhood of zero as time $t$ becomes sufficiently large, then $\hat{L}^{0}$ and $\hat{F}^{0}$ can be chosen as the initial admissible feedback gain matrices,  which also implies a proper selection of $\hat{P}^{0}$.
	We can also simply select $\hat{P}^{0}= cI_n$, where
	$c\ge 0$ is some given scalar. Initially, we can run Algorithm~\ref{alg:3} with $c= 0$. If Algorithm~\ref{alg:3} does not converge to a positive-definite matrix solution, then, increase $c$ gradually until the Algorithm~\ref{alg:3} converges to a positive-definite matrix solution.    It
	should be pointed out that the methods presented here for the choice of $\hat{P}^{0}$  in Algorithm~\ref{alg:3} are on the basis of experience, and this
	issue will also be pursued in future work.	
\end{Remark}
	\section{Numerical simulation}
 In this section, the performance of the proposed model-free Algorithm~\ref{alg:3} is investigated and compared with the model-based Algorithm 1. 
Consider a deterministic  continuous-time system model
of the F-16 aircraft plant studied in \cite{stevens2015aircraft}, and assume that
it is perturbed by state-dependent multiplicative noise. Then the perturbed F-16 aircraft plant can be described by system~(\ref{e1}) with matrices $A,B,E$ given in \cite{stevens2015aircraft} and 
\begin{equation*}
	A_{1}=\left[\begin{array}{ccc}
		0 & 0 & 0 \\
		-0.25 & 0.25 & 0\\
		0 & 0 & 0 	
	\end{array}\right].	
\end{equation*}
The performance index coefficients  are selected as $Q=I_{3}$, $R=I_{1}$ and $\gamma=5$.  

It is worthy pointing out that Algorithm~\ref{alg:3} is applied without knowing all the information about system~(\ref{e1}). We  pick  $\hat{L}^0$ and $\hat{F}^0$ at random to apply to the system
and  observe the  trend of state trajectories when time $t$ becomes sufficiently large
 to find the initial admissible feedback gains. We find $\hat{L}^0=[0.3976, -1.1913, 0.6625]$ and $\hat{F}^0=[0.4749, 1.3719, 0.3130]$ can make the state trajectories tend to  a neighborhood of zero, therefore, we choose them as the initial
feedback gains.  Set  the initial system state to $x_0=[0.1, 0.1, 0.1]^{\top}$ and then apply the chosen $\hat{L}^0$ and $\hat{F}^0$ with exploration noises to system~(\ref{e1}) to generate $\bar L=50$ sample paths for data collection.  Set  the length of integral interval
 to $\Delta t\ = 0.5$ and divide each integration interval $\bar K=100$ equal parts. Using the data collected,   data matrices  $\hat{\Phi}^{i}$ and  $\hat{\Upsilon}^{i}$ are calculated and then Algorithm~\ref{alg:3} is implemented.  The algorithm is terminated after  $N = 20$ iterations and then we use the results in the last iteration of Algorithm~\ref{alg:3} as the estimation of $P^*,L^*,F^*$, the corresponding state trajectories are shown in Fig.~1. 
\begin{figure}[htb]
	\begin{center}
		\includegraphics[height=4cm]{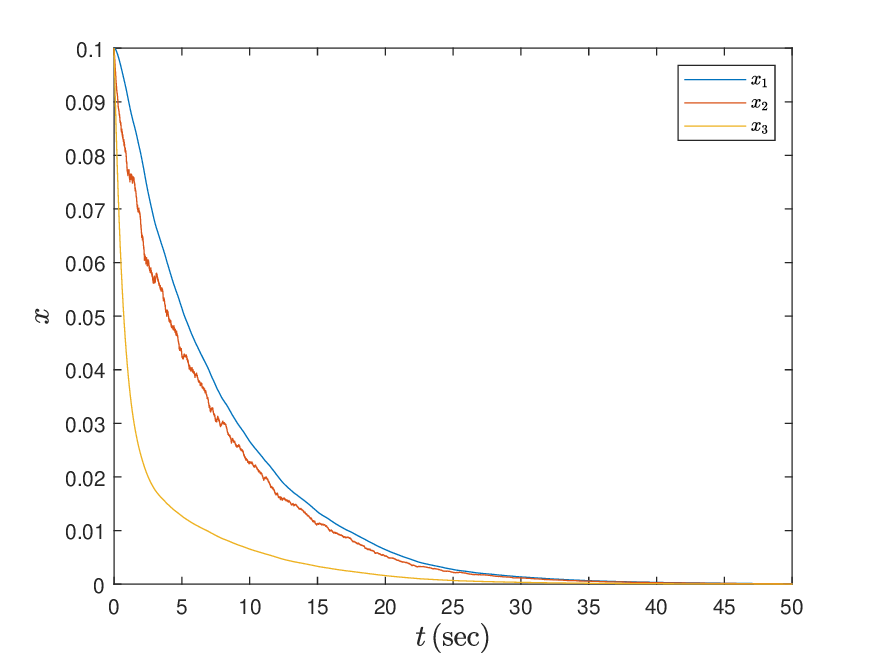}    
		\caption{State trajectories of the closed-loop F-16 aircraft system.}  
		\label{fig3}                                 
	\end{center}                                 
\end{figure}
 
 Now, we compare the model-free Algorithm~\ref{alg:3} with the model-based  Algorithm~\ref{alg:1}.  According to Theorem~\ref{th2},  we perform the model-based Algorithm \ref{alg:1} for a sufficiently large number of iterations and
 utilize the results in the last iteration 
 \begin{equation*}
 P^{N+1}=\left[\begin{array}{ccc}
 1.6908&1.3700&-0.1647\\
 1.3700&1.6833&-0.1816\\
 -0.1647&-0.1816&0.4372
 \end{array}\right]	
 \end{equation*}
 as an approximate value of $P^*$.
   To check whether $P^{N+1}$ is the solution of GARE~(\ref{gare}), $\mathscr{F}(P^{i+1})$ which is defined in (\ref{eq7})  is used to determine the
   distance from $P^{N+1}$ to the true solution $P^{*}$ of GARE~(\ref{gare}). When we insert $P^{N+1}$ into (\ref{eq7}), we get
\begin{equation*}
\mathscr{F}(P^{N+1})=10^{-16}\times\left[\begin{array}{ccc}
-4.4409&	6.8001&	0.0173\\
2.6368&	1.1102&	-1.3010\\
-1.0929&	0.8153&	1.1102
\end{array}\right].	
\end{equation*}
Since $\left\|\mathscr{F}(P^{N+1})\right\|_{F}=10^{-16}\times8.8844$ is small enough, we can utilize the results in the last iteration 
as  the optimal values $P^*$,$L^*$ and $F^*$. That is, they can be viewed as the benchmark for the simulation results of Algorithm~\ref{alg:3}.  Fig.~2 depicts the trajectories of the norms of the differences between $\hat P^i,\hat L^i,\hat F^i$ produced by Algorithm~\ref{alg:3}
and the optimal values  $P ^*,L^*,F^*$.
\begin{figure}[htb]
	\begin{center}
		\includegraphics[height=4cm]{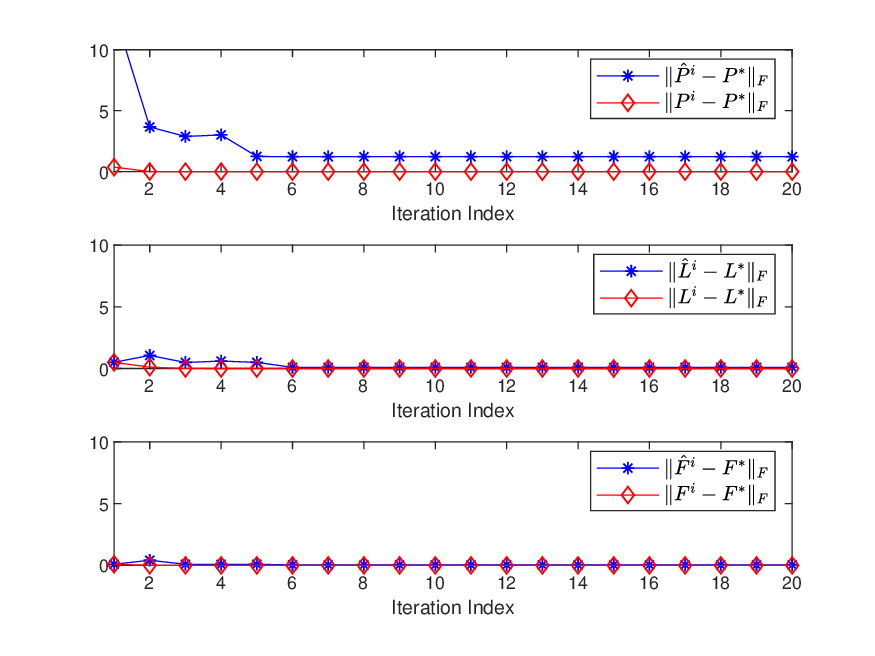}    
		\caption{The norms obtained by Algorithm~\ref{alg:1} and Algorithm~\ref{alg:3}.}  
		\label{fig2}                                 
	\end{center}                                 
\end{figure}

 The above comparison shows that, although the  error $\hat M^i$ induced by the unmeasurable stochastic noise in the system dynamics distorts  the
trajectories generated by Algorithm~\ref{alg:3} from the precise trajectories generated by model-based Algorithm~\ref{alg:1}, Algorithm~\ref{alg:3} still  successfully finds a near-optimal  control policy. This corresponds to the convergence conclusions achieved in Theorem~\ref{th6}.
\section{Conclusions}
	A data-driven off-policy RL method has been developed to
solve  stochastic $H_{\infty}$ control problem of continuous-time It\^o system with
unknown system models. Based on the model-based
SPU algorithm, an off-policy RL method is derived, which can learn
the solution of GARE from the system data generated
by arbitrary control and disturbance signals. The effectiveness of the proposed model-free off-policy RL method is verified by a stochastic linear F-16 aircraft  system. 
\section*{Appendix}
\appendix
\section{Proof of Theorem~\ref{th3}}
\begin{proof}\rm
	(\romannumeral1) Since $P^{*} \succeq 0$ is the stabilizing solution to  GARE~(\ref{gare}), one has $\sigma\left( \mathscr{L}_{\overline{A}^{*},A_{1}}\right) \subset \mathcal{C}_{-}$ according to Lemma~\ref{lem1}, where $\overline{A}^{*} :=\mathscr{A}(P^{*})$. It was shown in \cite{horn2012matrix} that the eigenvalues of a square matrix depend
	on the matrix elements continuously, thus there always exists a $\bar{\delta}_{0}>0$ such that $\sigma\left( \mathscr{L}_{\overline{A}^{i},A_{1}}\right) \subset \mathcal{C}_{-}$ for all $P^{i} \in \overline{\mathcal{B}}_{\bar{\delta}_{0}}\left(P^{*}\right)$, where $\overline{\mathcal{B}}_{\bar{\delta}_{0}}\left(P^{*}\right)$ is the closure of $\mathcal{B}_{\bar{\delta}_{0}}\left(P^{*}\right)$. 	
	
	(\romannumeral2) Suppose $P^{i} \in \overline{\mathcal{B}}_{\bar{\delta}_{0}}\left(P^{*}\right)$ for all $i \in \mathcal{Z}_{+}$.
	According to (\ref{eq6}), the sequence
	$\left\lbrace P^{i}\right\rbrace_{i=0}^{\infty} $ generated by Algorithm~\ref{alg:1} satisfies
	\begin{equation}\label{lp1}
	\mathscr{L}_{\overline{A}^{i},A_{1}}\left(P^{i+1}\right)=-\overline{Q}^{i}.
	\end{equation}
	Based on item 	(\romannumeral1)  of Theorem~\ref{th3}, we have that $\sigma\left( \mathscr{L}_{\overline{A}^{i},A_{1}}\right)$ does not contain zero eigenvalues. According to Lemma~\ref{lem2}, $\mathcal{H}\left(\overline{A}^{i},A_{1},H_{n}\right)$ is invertible, and then  GLE
	(\ref{eq6}) has a unique real symmetric solution (\ref{pp1}).
	
	Subtracting $P^{*} B R^{-1} B^{\top} P^{i}-\gamma^{-2} P^{*} EE^{\top} P^{i}+P^{i} B R^{-1}\\ B^{\top} P^{*}-\gamma^{-2} P^{i}EE^{\top} P^{*}$ from the both sides of GARE~(\ref{gare}) with $P=P^{*}$, one has
	\begin{equation*}
	\mathscr{L}_{\overline{A}^{i},A_{1}}\left(P^{*}\right)= {P}^{*} \overline{A}^{i}+\left( \overline{A}^{i}\right) ^{\top} {P}^{*}+A_{1}^{\top} {P}^{*} A_{1}=-\overline{Q}^{*},
	\end{equation*}
	where 
	\begin{equation*}
	\begin{aligned}
	\overline{Q}^{*}= & Q+P^{i} B R^{-1} B^{\top} P^{i} -\gamma^{-2} P^{i}EE^{\top} P^{i}\\
	& -\left(P^{i}-P^{*}\right) B R^{-1} B^{\top}\left(P^{i}-P^{*}\right)\\
	&+\gamma^{-2} \left(P^{i}-P^{*}\right)EE^{\top}\left(P^{i}-P^{*}\right).
	\end{aligned}
	\end{equation*}
	Similar to the process above, we have
	\begin{equation}\label{lp4}
	P^{*}=\left( \mathscr{L}_{\overline{A}^{i},A_{1}}\right)^{-1} \left( -\overline{Q}^{*}\right).
	\end{equation}
	Subtracting (\ref{lp4}) from (\ref{pp1}), we have
	\begin{equation*}\begin{array}{l} 
	\operatorname{vecs}\left( P^{i+1}-P^{*}\right)\\
	=\mathcal{H}\left(\overline{A}^{i},A_{1},H_{n}\right)^{-1}
	\operatorname{vecs}\left(\gamma^{-2}\left(P^{i}-P^{*}\right) EE^{\top}\left(P^{i}-P^{*}\right)\right.\\
	\left.-\left(P^{i}-P^{*}\right) B R^{-1} B^{\top}\left(P^{i}-P^{*}\right)\right).
	\end{array}		
	\end{equation*}
	Taking the Euclidean norm for vectors on both sides of the above equation, we have
	\begin{equation*}
	\begin{array}{l} 
	\left\|P^{i+1}-P^{*}\right\|_{F}
	\leq\left\|\mathcal{H}\left(\overline{A}^{i},A_{1},H_{n}\right)^{-1}\right\|_{F}\\
	\left\|P^{i}-P^{*}\right\|_{F}^{2}
	\left[ \gamma^{-2}\left\|EE^{\top}\right\|_{F}+ \left\|B R^{-1} B^{\top}\right\|_{F}\right] .	
	\end{array}	
	\end{equation*}
	By the invertibility of $\mathcal{H}\left(\overline{A}^{i},A_{1},H_{n}\right)$, there exists a $k_{1}>0$ such that  $\left\|\mathcal{H}\left(\overline{A}^{i},A_{1},H_{n}\right)^{-1}\right\|_{F} \leq k_{1}$  for all  $P^{i} \in \overline{\mathcal{B}}_{\bar{\delta}_{0}}\left(P^{*}\right)$. Then (\ref{th22}) is proved with $k_{0}=k_{1}\left(\gamma^{-2}\left\|EE^{\top}\right\|_{F}+ \left\|B R^{-1} B^{\top}\right\| _{F}\right)$. For any $\varepsilon_{0} \in(0,1)$, there exists a $\delta_{0}\in \left(0, \bar{\delta}_{0}\right]$ such that $k_{0} \delta_{0} \leq \varepsilon_{0}$, which proves (\ref{th23}).	\qed
\end{proof}
\section{Proof of Theorem~\ref{th6}}
We show several supplementary lemmas before showing Theorem~\ref{th6}.
\begin{lemma}\rm \label{lem9}
	For all $\hat P^{i} \in \mathcal{B}_{\delta_0}\left(P^{*}\right),i \in \mathcal{Z}_{+}$, there exists a $d(\delta_0) > 0$ that is not dependent on  $ \hat{P}^{i}$, such that if $\|\Delta M^{i}\|_{F}\leq d$, we have that 
	$\sigma\left( \mathscr{L}_{A-B\hat L^{i}-E\hat F^{i},A_{1}}\right) \subset \mathcal{C}_{-}$ and that $\left[  \hat M^{i}\right]  _{22}$ and $\left[  \hat M^{i}\right]_{33}$ are invertible.	
\end{lemma}
\begin{proof}\rm
	By the same route as item (\romannumeral1) in Theorem~\ref{th3}, it can be shown that  $\sigma\left( \mathscr{L}_{\mathscr{A}(\hat P^i),A_{1}}\right) \subset \mathcal{C}_{-}$ for all $\hat P^{i} \in \overline{\mathcal{B}}_{\bar{\delta}_{0}}\left(P^{*}\right)$. Because $\mathscr{A}(P)$ is a continuous function of $P$ and $\overline{\mathcal{B}}_{\bar{\delta}_{0}}\left(P^{*}\right)$ is a compact set,
	the set $\mathcal{A}:= \left\lbrace \mathscr{A}(\hat P^i)|\hat P^i \in \overline{\mathcal{B}}_{\bar{\delta}_{0}}\left(P^{*}\right)\right\rbrace $ is also compact. According to the continuity, for each $\hat P^{i} \in \overline{\mathcal{B}}_{\bar{\delta}_{0}}\left(P^{*}\right)$, there exists a constant $r> 0$ that depends on $\mathscr{A}(\hat P^i)$, such that $\sigma\left( \mathscr{L}_{Y,A_{1}}\right) \subset \mathcal{C}_{-}$ for any $Y \in \mathcal{B}_{r}\left( \mathscr{A}(\hat P^i)\right)$. According to the compactness of $\mathcal{A}$, for all $\hat P^{i} \in \overline{\mathcal{B}}_{\bar{\delta}_{0}}\left(P^{*}\right)$, there exists a $\bar r > 0$ that independs on $\mathscr{A}(\hat P^i)$, such that each  $Y \in \mathcal{B}_{\bar r}\left(\mathscr{A}\left(\hat{P}^{i}\right)\right)$
	satisfies  $\sigma\left( \mathscr{L}_{Y,A_{1}}\right) \subset \mathcal{C}_{-}$. It is worth noting that in (\ref{lf61}) and (\ref{lf62}), the improved policies  $\hat L^{i}$ and $\hat F^{i}$ are continuous functions of $\hat M^{i}$. Hence there exists a $d_{1}>0$ such that if  $\left\|\Delta M^{i}\right\|_{F} \leq d_{1}$ hods, one has $\left(A-B \hat{L}^{i}-E \hat{F}^{i}\right) \in \mathcal{B}_{\bar R}\left(\mathscr{A}\left(\hat{P}^{i}\right)\right)$, and furthermore,  $\sigma\left( \mathscr{L}_{A-B\hat L^{i}-E\hat F^{i},A_{1}}\right) \subset \mathcal{C}_{-}$  for all $\hat{P}^{i} \in \overline{\mathcal{B}}_{\bar{\delta}_{0}}\left(P^{*}\right)$. According to the continuity of the matrix inversion,
	there exists a $d_2 > 0$ such that  $\left[  \hat M^{i}\right]  _{22}$ and $\left[  \hat M^{i}\right]_{33}$ are invertible if $\|\Delta M^{i}\|_{F} \leq d_1$. Letting $d=\min \left(d_{1}, d_{2}\right)$ accomplishes the proof.\qed
\end{proof}
According to  Lemma~\ref{lem9}, if  $\|\Delta M\|_{l_{\infty}} \leq d$, then the sequence  $\left\{\hat{P}^{i}\right\}_{i=0}^{\infty}$  satisfies (\ref{hatp}). The next lemma provides
the upper bound of $\left\|\mathcal{D}\left(\overline{M}^{i},\Delta M^{i}\right)\right\|_{F}$.
\begin{lemma}\rm \label{lem10}
	For any $ k_{2}>0$,  there exist   $\delta_{1}^{1}\left(\delta_{0},k_{2}\right) \in (0,d]$ that is independent of $ \hat{P}^{i}$ and $k_{3}\left(\delta_{0}\right)>0$, such that if $\left\|\Delta M^{i}\right\|_{F}<\delta_{1}^{1}$, one has	
	\begin{equation*}
		\left\|\mathcal{D}\left(\overline{M}^{i}, \Delta M^{i}\right)\right\|_{F} \leq k_{3}\left\|\Delta M^{i}\right\|_{F}<k_{2}
	\end{equation*}	
	for all $\hat{P}^{i} \in \mathcal{B}_{\delta_{0}}\left(P^{*}\right)$, where $d$ is defined in Lemma~\ref{lem9}. 
\end{lemma}	
\begin{proof}\rm
	For all $\hat{P}^{i} \in \mathcal{B}_{\delta_{0}}\left(P^{*}\right) $ and  $\left\|\Delta M^{i}\right\|_{F} \leq d, i \in \mathcal{Z}_{+}$, according to the continuity of the matrix norm, Proposition
	\ref{pro:1} and Lemma~\ref{lem9}, we have  
	\begin{equation}\label{eq211}
		\begin{array}{l} 
			\left\|R^{-1} B^{\top} \hat{P}^{i}-\hat{L}^{i}\right\|_{F}
			\leq\left\| \left[\overline{M}^{i}\right]_{22} ^{-1}\right\|_{F}\\ \left(\left\|\Delta M^{i}\right\|_{F}+\left\|\left[\hat{M}^{i}\right]_{22}^{-1}\right\|_F\left\|\left[\hat{M}^{i}\right]_{21}\right\|_F\left\|\Delta M^{i}\right\|_{F}\right) \\
			\leq k_{4}^1\left\|\Delta M^{i}\right\|_{F}
		\end{array}	
	\end{equation}
	and
	\begin{equation}\label{eq212}
		\begin{array}{l} 
			\left\|-\gamma^{-2}  E^{\top}  \hat{P}^{i}-\hat{F}^{i}\right\|_{F}
			\leq\left\| \left[\overline{M}^{i}\right]_{33} ^{-1}\right\|_{F}\\
			\left(\left\|\Delta M^{i}\right\|_{F}+\left\|\left[\hat{M}^{i}\right]_{33}^{-1}\right\|_F\left\|\left[\hat{M}^{i}\right]_{31}\right\|_F\left\|\Delta M^{i}\right\|_{F}\right) \\
			\leq k_{4}^2\left\|\Delta M^{i}\right\|_{F}
		\end{array}
	\end{equation}		
	for some $k_{4}^1\left(\delta_{0}, d\right)>0,k_{4}^2\left(\delta_{0}, d\right)>0$. Define 
	\begin{equation*}
		\begin{aligned}
			\Delta X_i= &	\mathcal{H}\left(\hat{\overline{A}}^{i},A_{1},H_{n}\right)-\mathcal{H}\left(A-B\hat L^{i}-E\hat F^{i},A_{1},H_{n}\right),\\
			\Delta Y_i =&	\operatorname{vecs}\left(\hat{Q}^{i}_2-\hat{Q}^{i}_1\right).
		\end{aligned}	
	\end{equation*}	
	Noting (\ref{eq211}) and (\ref{eq212}), it is easy to check that $\left\|\Delta X_i\right\|_{F} \leq k_{5}\left\|\Delta M^{i}\right\|_{F}, \left\|\Delta Y_i\right\|_{2} \leq k_{6}\left\|\Delta M^{i}\right\|_{F}	$ for some $k_{5}\left(\delta_{0}, d\right)>0 $ and  $k_{6}\left(\delta_{0}, d\right)>0$. Then  according to the continuity of the matrix norm, Lemma~\ref{lem2} and Proposition
	\ref{pro:1}, one has 
	\begin{equation*}
		\begin{array}{l} 	
			\left\|\mathcal{D}\left( \overline{M}^{i+1}, \Delta M^{i+1}\right)\right\|_{F}\\ 
			\leq \left\|\mathcal{H}\left(A-B\hat L^{i}-E\hat F^{i},A_{1},H_{n}\right)^{-1}\right\|_{F}\left\|\Delta M^{i}\right\|_{F}\\
			\left\lbrace k_{6}+k_{5}\left\|\mathcal{H}\left(\hat{\overline{A}}^{i},A_{1},H_{n}\right)^{-1}\right\|_{F}\right.\\
			\left. \left[ \left\|-\hat P^{i} B R^{-1} B^{\top} \hat P^{i}\right\|_{F} 
			+\left\|\gamma^{-2} \hat P^{i}E E^{\top} \hat P^{i}\right\|_{F}\right] \right\rbrace 
			\\
			\leq k_{3}\left(\delta_{0}\right)\left\|\Delta M^{i}\right\|_{F}.
		\end{array}
	\end{equation*}	
	Taking  $0<\delta_{1}^{1} \leq d $ with  $k_{3} \delta_{1}^{1}<k_{2}$ completes the proof.	\qed
\end{proof}
Now we are in a position to prove Theorem~\ref{th6}.
\begin{proof}\rm Let  $k_{2}=(1-\varepsilon_{0}) \delta_{0}$  in Lemma~\ref{lem10} and  $\delta_{1}=\delta_{1}^{1}\left(\delta_{0},k_{2}\right)$. For any  $i \in \mathcal{Z}_{+}$, if  $\hat{P}^{i} \in \mathcal{B}_{\delta_{0}}\left(P^{*}\right)$, then
	\begin{align}
		&\left\|\hat{P}^{i+1}-P^{*}\right\|_{F}  \notag\\
		\leq& \left\|\left( \mathscr{L}_{\hat{\overline{A}}^{i},A_{1}}\right)^{-1} \left( -\hat{Q}^{i}_1\right) -P^{*}\right\|_{F}+\left\|\mathcal{D}\left(\overline{M}^{i}, \Delta M^{i}\right)\right\|_{F} \notag\\
		\leq& \varepsilon_{0}\left\|\hat{P}^{i}-P^{*}\right\|_{F}+k_{3}\left\|\Delta M^{i}\right\|_{F} \label{eq221}\\
		\leq &\varepsilon_{0}\left\|\hat{P}^{i}-P^{*}\right\|_{F}+k_{3}\|\Delta M\|_{l_{\infty}} \label{eq231}\\
		<	&\varepsilon_{0} \delta_{0}+k_{3} \delta_{1}<\varepsilon_{0} \delta_{0}+k_{2}=\delta_{0},\label{eq241}
	\end{align}	
	where (\ref{eq221}) and (\ref{eq241}) hold because of  Theorem~\ref{th3} and Lemma~\ref{lem10}. By induction,  (\ref{eq221}) --  (\ref{eq241}) hold for all  $i \in \mathbb{Z}_{+}$, therefore (\romannumeral1) in Theorem~\ref{th6} is proved.  As a result,  according to (\ref{eq221}), one has
	\begin{equation*}
		\begin{aligned}
			&\left\|\hat{P}^{i}-P^{*}\right\|_{F}\\
			\leq &\varepsilon_{0}^{2}\left\|\hat{P}^{i-2}-P^{*}\right\|_{F}+(\varepsilon_{0}+1) k_{3}\|\Delta M\|_{l_{\infty}} \\
			\leq& \cdots\\
			\leq& \varepsilon_{0}^{i}\left\|\hat{P}^{0}-P^{*}\right\|_{F}+\left(1+\cdots+\varepsilon_{0}^{i-1}\right) k_{3}\|\Delta M\|_{l_{\infty}} \\
			<&\varepsilon_{0}^{i}\left\|\hat{P}^{0}-P^{*}\right\|_{F}+\frac{k_{3}}{1-\varepsilon_{0}}\|\Delta M\|_{l_{\infty}},	
		\end{aligned}	
	\end{equation*}	
	which proves 	(\romannumeral2) in Theorem~\ref{th6}.
	As to	(\romannumeral3) in Theorem~\ref{th6}, for any  $\varepsilon>0 $, there exists a  $i_{1} \in \mathcal{Z}_{+}$ such that $ \sup \left\{\left\|\Delta M^{i}\right\|_{F}\right\}_{i=i_{1}}^{\infty}<\beta^{-1}(\varepsilon / 2)$. Let  $i_{2} \geq i_{1}$, for $ k \geq i_{2}$, because $\hat{P}^{i}$ is bounded,  in accordance with (\romannumeral2) in Theorem~\ref{th6}, we have
	\begin{equation*}
		\begin{aligned}
			\left\|\hat{P}^{i}-P^{*}\right\|_{F} & \leq \alpha\left(\left\|\hat{P}^{i_{2}}-P^{*}\right\|_{F}, i-i_{2}\right)+\varepsilon / 2 \\
			& \leq \alpha \left( k_{7}, i-i_{2}\right)+\varepsilon / 2,
		\end{aligned}
	\end{equation*}
	for some $k_{7}\left(\delta_{0}, d\right)>0 $.  Because $ \lim _{i \rightarrow \infty} \alpha\left(k_{7}, i-i_{2}\right)=0$, there is a $ i_{3} \geq   i_{2}$ such that $ \alpha\left(k_{7}, i-i_{2}\right)<\varepsilon / 2$ for all $i \geq i_{3}$, which completes the proof.\qed
\end{proof}
\bibliographystyle{apacite}   
\bibliography{gjref2}           
\end{document}